\documentclass[final,10pt]{elsarticle}
\usepackage{lineno,hyperref}
\usepackage{epsfig,amssymb,amsmath,version,amsthm}
\usepackage{amssymb,version,graphicx,fancybox,mathrsfs,multirow}
\usepackage[notcite,notref]{showkeys}
\usepackage{url,hyperref}
\usepackage{subfigure,epstopdf,bm}
\usepackage{color, xcolor}
\usepackage{appendix}
\usepackage{algorithm}
\usepackage{algorithmicx}
\usepackage{algpseudocode}
\usepackage{amsmath}
\usepackage{amsfonts}
\usepackage{amssymb}
\usepackage{graphicx}%

\modulolinenumbers[5]

\journal{Journal of Computational Physics}

\catcode`\@=11 \theoremstyle{plain}
\@addtoreset{equation}{section}   

\@addtoreset{figure}{section}
\renewcommand\thefigure{\thesection.\@arabic\c@figure}
\newtheorem{thm}{\bf Theorem}

\newtheorem{proposition}{Proposition}[section]
\newenvironment{theorem}{\begin{thm}} {\end{thm}}
\newtheorem{cor}{\bf Corollary}

\newtheorem{lmm}{\bf Lemma}

\newenvironment{lemma}{\begin{lmm}}{\end{lmm}}
\theoremstyle{remark}
\newtheorem{rem}{Remark}[section]

\def \ri {{\rm i}}
\def \widebar {\accentset{{\cc@style\underline{\mskip10mu}}}}

\newcommand{\bs}[1]{\boldsymbol{#1}}

\begin{document}

\begin{frontmatter}
\title{Fast Multipole Method for 3-D Linearized Poisson-Boltzmann Equation in Layered Media}

\author[fn1,fn2]{Bo Wang}
\author[fn2]{Wen Zhong Zhang}
\author[fn2]{Wei Cai\corref{cor1}}
\ead{cai@mail.smu.edu}

\cortext[cor1]{Corresponding author}

\address[fn1]{LCSM(MOE), School of Mathematics and Statistics, Hunan Normal University, Changsha, Hunan, 410081, P. R. China.}
\address[fn2]{Department of Mathematics, Southern Methodist University, Dallas, TX 75275, USA.}

\begin{abstract}
	In this paper, we propose a fast multipole method (FMM) for 3-D linearized Poisson-Boltzmann (PB) equation in layered media. The main framework of the algorithm is analogous to the FMM for Helmholtz and Laplace equation in layered media \cite{wang2019fast, wang2019fastlaplace}, using an extension of the Funk-Hecke formula for pure imaginary wave number. Moreover, a recurrence formula is provided for the run-time computation of the Sommerfeld-type integrals used in the FMM algorithm. Due to the similarity between Helmholtz and linearized PB equation, the recurrence formula can also be used for the FMM of Helmholtz equation in layered media with minor changes as mentioned in \cite{wang2019fast}. Numerical results validate that the FMM for interactions of charges under screen's potentials in layered media has  the same accuracy and CPU complexity as the classic FMM for charge interactions in free space.
\end{abstract}
\begin{keyword}
	Fast multipole method, Poisson-Boltzmann equation, layered media, spherical harmonic expansion, equivalent polarization source
\end{keyword}
\end{frontmatter}

\section{Introduction}
In this paper, we continue our research on the FMM associated with the Green's function of elliptic equations
\begin{equation}\label{PBequation}
a_{\ell}\big[\boldsymbol{\Delta}u_{\ell\ell'}(\boldsymbol{r},\boldsymbol{r}^{\prime
})+\kappa_{\ell}^2 u_{\ell\ell'}(\boldsymbol{r},\boldsymbol{r}^{\prime
})\big]=-\delta(\boldsymbol{r},\boldsymbol{r}^{\prime}),
\end{equation}
in layered media, where $\{a_{\ell}, \kappa_{\ell}\}$ are parameters in the $\ell$-th layer, $\delta(\boldsymbol{r},\boldsymbol{r}^{\prime})$ is the
Dirac delta function and $\bs r$ and $\bs r'$ are two points in the $\ell$ and $\ell'$-th layers, respectively. More general settings about the layered structure will be discussed later. Three typical cases with parameters $\{a_{\ell}, \kappa_{\ell}\}$ to be $\{1, k_{\ell}\}, \{\varepsilon_{\ell}, 0\}$ or $\{\varepsilon_{\ell},\ri\lambda_{\ell}\}$, such that $\varepsilon_{\ell}, k_{\ell}, \lambda_{\ell}>0$ are related to three categories of important applications. Typical examples among them are wave propagation, static electromagnetics and nuclear physics. Recently, we have developed FMM for Helmholtz and Laplace equations (cf. \cite{wang2019fast, wang2019fastlaplace}) which are corresponding to the cases with parameters $\{a_{\ell}=1, \kappa_{\ell}=k_{\ell}\}$ and $\{a_{\ell}=\varepsilon_{\ell}, \kappa_{\ell}=0\}$. Now, we continue our work to consider the linearized Poisson-Boltzmann equation, i.e., the case $\{a_{\ell}=\varepsilon_{\ell}, \kappa_{\ell}=\ri\lambda_{\ell}\}$ with $\lambda_{\ell}>0$.

Problems modeling by linearized Poisson-Boltzmann equations arise in various applications in physics, chemistry and biology when Coulomb forces are damped by screening effects (cf. \cite{juffer1991electric, liang1997computation, lin2014accuracy,cai2013computational}). In the free space case, the Green's function is referred as Yukawa potential or screened Coulomb potential and the classic FMM for Coulomb potential has been successfully extended to Yukawa potential (cf. \cite{greengard2002new, huang2009fmm}) for the reduction of the $O(N^{2})$ cost of computing corresponding $N$ particles (or sources) problem to $O(N)$. Due to the significant efficiency improvement, Yukawa-FMM has been widely used in modern computational biology, chemistry (cf. \cite{lu2008recent, brown2011implementing}).

Similar with the FMM for Coulomb potential \cite{greengard1987fast,greengard1997new}, the mathematical foundation of the Yukawa-FMM is the addition theorem for modified Bessel functions as we will review in the next section. The theory provides optimal approximation for the far field interactions in the free space, and at the meantime implies the major difficulty for the development of FMM in layered media. That is optimal expansion theory for the Green's function in layered media (layered Green's function in short in this paper) has not been developed. Nevertheless, layered media are
usually the basic setting in many applications in biology, e.g., ion channel \cite{lin2014accuracy}. For those applications, layered Green's function is preferably used to describe the interactions to avoid introducing artificial unknowns on the infinite material interfaces. To handle
the interaction of sources embedded in layered media using layered Green's
functions, various approaches have been proposed (cf. \cite{cho2018heterogeneous,wang2020taylor,tausch2003fast,cho2012parallel,ying2004kernel,wang2019kernel}).

Recently, we have proposed a mathematical theory to obtain optimal far field approximation for the layered Green's function of Helmholtz and Laplace equations (cf. \cite{wang2019fast, zhang2018exponential,wang2019fastlaplace}), where the generating function of the Bessel function (2-D case) or a Funk-Hecke formula (3-D case) were used to connect Bessel functions and plane wave functions. With the optimal far field approximation theory, corresponding FMMs for layered Green's function have been implemented. The reason of using Fourier (2-D case) and spherical harmonic (3-D case) expansions of plane waves is that the layered Green's functions have Sommerfeld-type integral representations in which the plane waves are involved. Note that the layered Green's function of the linearized Poisson-Boltzmann equation also has similar integral form as that of Helmholtz equation, we will continue the series research work by investigating the case with pure imaginary parameters $\kappa_{\ell}=\ri\lambda$. Another extension the Funk-Hecke formula with pure imaginary wave number is the key in the derivation of the multipole and local expansions (MEs, LEs) and multipole to local (M2L) translation operators for the reaction components of the layered Green's function. Under the framework proposed in our previous work, the potential due to sources embedded in layered media is decomposed into free space and reaction components and equivalent polarization charges are introduced to re-express the reaction components. The FMM in layered media will then consist of classic Yukawa-FMM for the free space components and FMMs for reaction components based on equivalent polarization sources and the new MEs, LEs and M2L translations. Moreover, in order to avoid making memory consuming pre-computed 3-D tables (cf. \cite{wang2019fast}), we will develop a recurrence formula for efficient computation of the Sommerfeld-type integrals used in the algorithm. The FMMs for the reaction field components are much faster than that for the free space components due to the fact that the introduced equivalent polarization charges are always separated from the associated target charges by a material interface. As a result, the new FMM for sources in layered media costs almost the same as the Yukawa-FMM for the free space case.

The rest of the paper is organized as follows. In Section 2, after a short discussion
on the Green's function in layered media consisting of free
space and reaction components, we present the formulas for the potential induced by sources embedded in layered media. In section 3, we first further extend the Funk-Hecke formula to derive a spherical harmonic expansion for the exponential functions involved in the integral representation of the layered media Green's function of Poisson-Boltzmann equation. By using this expansion, we present a new approach for the derivation of the ME, LE and M2L operators of the free space Green's function. The same approach will be then used to derive MEs, LEs and M2L translation operators for the reaction components of the layered Green's function.  By introducing equivalent polarization charge for each type of the reaction components, the reaction potentials are re-expressed. Then the MEs, LEs and M2L translation operators for the reaction
components are derived based on the new expressions. Combining the original source charges and the equivalent polarization charges associated to each reaction component, the FMMs for reaction components can be implemented.
Section 4 will give numerical results to show the spectral accuracy and $O(N)$ complexity of the proposed FMM for interactions in layered
media. Finally, a conclusion is given in Section 5.

\section{Potential due to sources in layered media}
In this section, the potential induced by sources embedded in layered media is formulated using layered Green's function and then decomposed into a free space and four types of reaction components.

\subsection{Green's function of linearized Poisson-Boltzmann equation in layered media} Consider a layered medium consisting of $L$-interfaces located at $z=d_{\ell
},\ell=0,1,\cdots,L-1$ in Fig. \ref{layerstructure}. The material parameters are given by $\{a_{\ell}=\varepsilon_{\ell},\kappa_{\ell}=\ri\lambda_{\ell}\}_{\ell=0}^L$ where $\varepsilon_{\ell}$ and $\lambda_{\ell}$ are the dielectric constant and the inverse Debye-Huckel length in the $\ell$-th layer.
\begin{figure}[ht!]\label{layerstructure}
	\centering
	\includegraphics[scale=0.7]{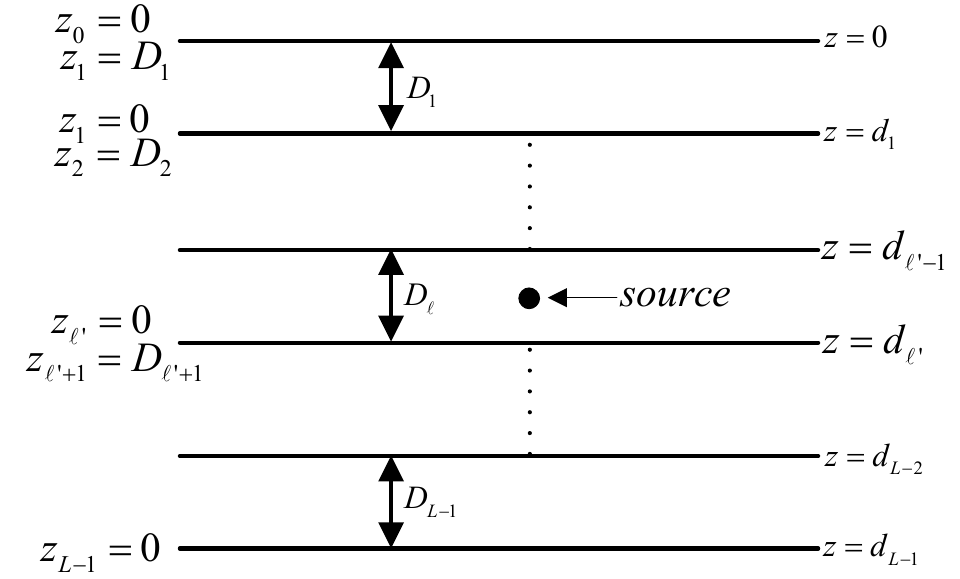}
	\caption{Sketch of the layer structure for general multi-layer media.}
\end{figure}
Suppose we have a point
source at $\boldsymbol{r}^{\prime}=(x^{\prime},y^{\prime},z^{\prime})$ in the
$\ell^{\prime}$-th layer ($d_{\ell^{\prime}}<z^{\prime}<d_{\ell^{\prime}-1}$). Then, the layered Green's function of the linearized PB equation satisfies \eqref{PBequation} at field point $\boldsymbol{r}=(x,y,z)$ in the $\ell$-th layer ($d_{\ell
}<z<d_{\ell}-1$) where $\delta(\boldsymbol{r},\boldsymbol{r}^{\prime})$ is the
Dirac delta function.
By using partial Fourier transform along $x-$ and $y-$directions, the problem can be solved analytically for each layer in $z$ by imposing
transmission conditions at the interface between $\ell$-th and $(\ell-1)$-th
layer ($z=d_{\ell-1})$, \textit{i.e.},
\begin{equation}
u_{\ell-1,\ell'}(x,y,z)=u_{\ell\ell'}(x,y,z),\quad \varepsilon_{\ell-1}\frac{\partial  u_{\ell-1,\ell'}(x,y,z)}{\partial z}=\varepsilon_{\ell}\frac{\partial \widehat u_{\ell\ell'}(k_{x},k_{y},z)}{\partial z},
\end{equation}
as well as decaying conditions in the top and bottom-most layers for
$z\rightarrow\pm\infty$.

Here, we just present the expression of layered Green's function, and the derivation is an analogue to that for layered Green's function of the Helmholtz equation (cf. \cite{wang2020taylor}). The expression of the layered Green's function in the physical domain takes the form
\begin{equation}\label{layeredGreensfun}
u_{\ell\ell^{\prime}}(\boldsymbol{r},\boldsymbol{r}^{\prime})=\begin{cases}
\displaystyle u_{\ell\ell'}^{\text{react}}(\boldsymbol{r},\boldsymbol{r}^{\prime})+\frac{e^{-\lambda_{\ell} |\bs r-\bs r^{\prime}|}}{4\pi\varepsilon_{\ell}
	|\bs r-\bs r^{\prime}|},&\ell=\ell',\\
\displaystyle u_{\ell\ell'}^{\text{react}}(\boldsymbol{r},\boldsymbol{r}^{\prime}), & \text{otherwise},
\end{cases}
\end{equation}
where $u^{\text{react}}_{\ell\ell'}(\bs r, \bs r')$ is the reaction component induced by the layered media. In general, $u^{\text{react}}_{\ell\ell'}(\bs r, \bs r')$ has two components. However, only one component left in the top and bottom layer due to decaying conditions as $z\rightarrow\pm\infty$. Thus, the reaction component has decomposition
\begin{equation}\label{reactioncomponent}
u^{\text{react}}_{\ell\ell'}(\bs r, \bs r')=\begin{cases}
\displaystyle u_{0\ell^{\prime}}^{1}%
(\boldsymbol{r},\boldsymbol{r}^{\prime}),\\
\displaystyle u_{\ell\ell^{\prime}}^{1}%
(\boldsymbol{r},\boldsymbol{r}^{\prime})+u_{\ell\ell^{\prime}}^{2
}(\boldsymbol{r},\boldsymbol{r}^{\prime}), &
0<\ell<L,\\
\displaystyle u_{L\ell^{\prime}}^{2}%
(\boldsymbol{r},\boldsymbol{r}^{\prime}),
\end{cases}
\end{equation}
with components given by Sommerfeld-type integrals:
\begin{equation}%
\begin{cases}
\displaystyle u_{\ell\ell'}^{1}(\bs r, \bs r')=\frac{1}{8\pi^2 }\int_0^{\infty}\int_0^{2\pi}\lambda_{\rho}e^{\ri\bs\lambda_{\alpha}\cdot(\bs\rho-\bs\rho')}\frac{e^{- \lambda_{\ell z} (z-d_{\ell})}}{\lambda_{\ell z}}\psi_{\ell\ell'}^{1}(\lambda_{\rho}, z')d\alpha d\lambda_{\rho},\quad \ell<L,\\[8pt]
\displaystyle u_{\ell\ell'}^{2}(\bs r, \bs r')=\frac{1}{8\pi^2}\int_0^{\infty}\int_0^{2\pi}\lambda_{\rho}e^{\ri\bs\lambda_{\alpha}\cdot(\bs\rho-\bs\rho')}\frac{e^{- \lambda_{\ell z} (d_{\ell-1}-z)}}{\lambda_{\ell z}}\psi_{\ell\ell'}^{2}(\lambda_{\rho}, z')d\alpha d\lambda_{\rho},\quad \ell>0,
\end{cases}
\label{greenfuncomponent}%
\end{equation}
where $\bs\lambda_{\alpha}=\lambda_{\rho}(\cos\alpha,\sin\alpha)$, $\bs\rho=(x, y)$, $\bs \rho'=(x', y')$,
\begin{equation}%
\begin{split}
\psi_{\ell0}^{1}(\lambda_{\rho},z^{\prime})=\begin{cases}
\displaystyle e^{-
	\lambda_{\ell'z}z^{\prime}}\sigma_{\ell0}^{11}(k_{\rho
}),\\[8pt]%
\displaystyle e^{-\lambda_{\ell'z}(z^{\prime}-d_{\ell^{\prime}})}\sigma_{\ell
	\ell^{\prime}}^{11}(\lambda_{\rho})+e^{-\lambda_{\ell'z}(d_{\ell^{\prime}-1}-z^{\prime})}\sigma_{\ell\ell^{\prime}}^{12}(\lambda_{\rho}),\quad0<\ell^{\prime}<L,\\[8pt]%
\displaystyle e^{-
	\lambda_{\ell'z}(d_{L-1}-z^{\prime})}\sigma_{\ell L}^{12
}(\lambda_{\rho}).
\end{cases}\\
\psi_{\ell\ell^{\prime}}^{2}(k_{\rho
},z^{\prime})=
\begin{cases}
\displaystyle e^{-
	\lambda_{\ell'z}z^{\prime}}\sigma_{\ell0}^{21}(k_{\rho
}),\\[8pt]%
\displaystyle e^{-\lambda_{\ell'z}(z^{\prime}-d_{\ell^{\prime}})}%
\sigma_{\ell\ell^{\prime}}^{21}(\lambda_{\rho})+e^{-
	\lambda_{\ell'z}(d_{\ell^{\prime}-1}-z^{\prime})}\sigma_{\ell\ell^{\prime}%
}^{22}(\lambda_{\rho}),\quad0<\ell^{\prime}<L,\\[8pt]%
\displaystyle e^{-
	\lambda_{\ell'z}(d_{L-1}-z^{\prime})}\sigma_{\ell L}^{22
}(\lambda_{\rho}).
\end{cases}
\end{split}
\label{totaldensity}%
\end{equation}
It is worthy to point out that reaction densities $\sigma_{\ell\ell'}^{11}(\lambda_{\rho}), \sigma_{\ell\ell'}^{12}(\lambda_{\rho}), \sigma_{\ell\ell'}^{21}(\lambda_{\rho}), \sigma_{\ell\ell'}^{22}(\lambda_{\rho})$ are only determined by the layered structure and the material parameters $\varepsilon_{\ell}$ and $\lambda_{\ell}$ in each layers. The above are general formulas which are applicable to multi-layered media. Here, we
give explicit formulas (see \eqref{densitythreelayer1},\eqref{densitythreelayer2},
\eqref{densitythreelayer3}) for reaction densities in the cases of three layers as example.
\begin{itemize}
	\item Source in the top layer:
	\begin{equation}\label{densitythreelayer1}
	\begin{split}
	\sigma_{00}^{11}(\lambda_{\rho})=&\frac{\varepsilon_1 \lambda_{1z} (\varepsilon_0 \lambda_{0z} - \varepsilon_2 \lambda_{2z}) \cosh(-d_1\lambda_{1z}) - (\varepsilon_1^2 \lambda_{1z}^2 - \varepsilon_0 \varepsilon_2 \lambda_{0z} \lambda_{2z}) \sinh(-d_1\lambda_{1z})}{\varepsilon_0\kappa(\lambda_{\rho})},\\
	\sigma_{10}^{11}(\lambda_{\rho})=&\frac{\varepsilon_0\lambda_{1z}(\lambda_1 \lambda_{1z}- \varepsilon_2 \lambda_{2z})}{\varepsilon_0\kappa(\lambda_{\rho})},\quad \sigma_{10}^{21}(\lambda_{\rho})=\frac{\varepsilon_0\lambda_{1z}(\varepsilon_1 \lambda_{1z}+ \varepsilon_2 \lambda_{2z}) e^{-d_1\lambda_{1z}}}{\varepsilon_0\kappa(\lambda_{\rho})},\\
	\sigma_{20}^{21}(\lambda_{\rho})=&\frac{2\varepsilon_0 \varepsilon_1 \lambda_{1z}\lambda_{2z} }{\varepsilon_0\kappa(\lambda_{\rho})}.
	\end{split}
	\end{equation}
	\item Source in the middle layer:
	\begin{equation}\label{densitythreelayer2}
	\begin{split}
	\sigma_{01}^{11}(\lambda_{\rho})=&\frac{\varepsilon_1\lambda_{0z}(\varepsilon_{1}\lambda_{1z}-\varepsilon_2\lambda_{2z}) }{\varepsilon_1\kappa(\lambda_{\rho})},\quad
	\sigma_{01}^{12}(\lambda_{\rho})=\frac{\varepsilon_1\lambda_{0z}(\varepsilon_{1}\lambda_{1z}+\varepsilon_2\lambda_{2z})e^{-d_1\lambda_{1z}}}{\varepsilon_1\kappa(\lambda_{\rho})},\\
	\sigma_{11}^{11}(\lambda_{\rho})=&\frac{(\varepsilon_1\lambda_{1z}-\varepsilon_2\lambda_{2z})(\varepsilon_1\lambda_{1z}+\varepsilon_0\lambda_{0z}) e^{-d_1 \lambda_{1z}}}{2\varepsilon_1\kappa(\lambda_{\rho})},\\
	\sigma_{11}^{12}(\lambda_{\rho})=&\frac{(\varepsilon_1\lambda_{1z}-\varepsilon_2\lambda_{2z})(\varepsilon_1\lambda_{1z}-\varepsilon_0\lambda_{0z}) }{2\varepsilon_1\kappa(\lambda_{\rho})},\\
	\sigma_{11}^{21}(\lambda_{\rho})=&\frac{(\varepsilon_1\lambda_{1z}-\varepsilon_2\lambda_{2z})(\varepsilon_1\lambda_{1z}+\varepsilon_0\lambda_{0z})}{2\varepsilon_1\kappa(\lambda_{\rho})},\\
	\sigma_{11}^{22}(\lambda_{\rho})=&\frac{(\varepsilon_1\lambda_{1z}+\varepsilon_2\lambda_{2z})(\varepsilon_1\lambda_{1z}-\varepsilon_0\lambda_{0z})e^{-d_1\lambda_{1z}} }{2\varepsilon_1\kappa(\lambda_{\rho})},\\
	\sigma_{21}^{21}(\lambda_{\rho})=&\frac{\varepsilon_1 \lambda_{2z}(\varepsilon_0\lambda_{0z}+\varepsilon_1\lambda_{1z})e^{-d_1\lambda_{1z}  } }{\varepsilon_1\kappa(\lambda_{\rho})},\quad\sigma_{21}^{22}(\lambda_{\rho})=\frac{\varepsilon_1 \lambda_{2z}(\varepsilon_1\lambda_{1z}-\varepsilon_0\lambda_{0z}) }{\varepsilon_1\kappa(\lambda_{\rho})}
	.
	\end{split}
	\end{equation}
	\item Source in the bottom layer:
	\begin{equation}\label{densitythreelayer3}
	\begin{split}
	\sigma_{02}^{12}(\lambda_{\rho})=&\frac{2\varepsilon_1\lambda_{1z}\varepsilon_2\lambda_{0z}}{\varepsilon_2\kappa(\lambda_{\rho})},\\
	\sigma_{12}^{22}(\lambda_{\rho})=&\frac{\varepsilon_2\lambda_{1z}(\varepsilon_1\lambda_{1z}-\varepsilon_0\lambda_{0z})}{\varepsilon_2\kappa(\lambda_{\rho})},\quad
	\sigma_{12}^{12}(\lambda_{\rho})=\frac{\varepsilon_2\lambda_{1z}(\varepsilon_0\lambda_{0z}+\varepsilon_1\lambda_{1z})e^{-d_1\lambda_{1z}} }{\varepsilon_2\kappa(\lambda_{\rho})},\\
	\sigma_{22}^{22}(\lambda_{\rho})=&\frac{\varepsilon_1 \lambda_{1z} (\varepsilon_2 \lambda_{2z}-\varepsilon_0 \lambda_{0z} ) \cosh(-d_1\lambda_{1z}) - (\varepsilon_1^2 \lambda_{1z}^2 - \varepsilon_0 \varepsilon_2 \lambda_{0z} \lambda_{2z}) \sinh(-d_1\lambda_{1z})}{\varepsilon_2\kappa(\lambda_{\rho})},
	\end{split}
	\end{equation}
\end{itemize}
where
$$\kappa(\lambda_{\rho})=\varepsilon_1 \lambda_{1z} (\varepsilon_0 \lambda_{0z} + \varepsilon_2 \lambda_{2z}) \cosh(-d_1\lambda_{1z}) + (\varepsilon_1^2 \lambda_{1z}^2 + \varepsilon_0 \varepsilon_2 \lambda_{0z} \lambda_{2z}) \sinh(-d_1\lambda_{1z}).$$

\subsection{Components of the potential due to sources embedded in layered media}
Let $\mathscr{P}_{\ell}=\{(Q_{\ell j},\boldsymbol{r}_{\ell j}),$ $j=1,2,\cdots
,N_{\ell}\}$, $\ell=0, 1, \cdots, L$ be $L$ groups of source particles distributed in a multi-layered medium with $L+1$ layers (see Fig. \ref{layerstructure}). The group of particles in $\ell$-th layer is denoted by $\mathscr{P}_{\ell}$.  Then, the potential at $\bs r_{\ell i}$ due to all other particles is given by  the summation
\begin{equation}\label{potential1}
\begin{split}
\Phi_{\ell}(\boldsymbol{r}_{\ell i})=&\sum\limits_{\ell'=0}^{L}\sum\limits_{j=1}^{N_{\ell'}}Q_{\ell' j}u_{\ell\ell'}(\bs r_{\ell i},\bs r_{\ell' j})\\
=&\sum\limits_{j=1,j\neq i}^{N_{\ell}}Q_{\ell j}\frac{e^{-\lambda_{\ell}|\bs r_{\ell i}-\bs r_{\ell j}|}}{4\pi\varepsilon_{\ell}|\bs r_{\ell i}-\bs r_{\ell j}|}+\sum\limits_{\ell'=0}^{L}\sum\limits_{j=1}^{N_{\ell'}}Q_{\ell' j}u_{\ell\ell'}^{\text{react}}(\bs r_{\ell i},\bs r_{\ell' j}).
\end{split}
\end{equation}
By expressions in \eqref{greenfuncomponent} and \eqref{totaldensity}, $u_{\ell\ell'}^{1}(\bs r, \bs r')$ and $u_{\ell\ell'}^{2}(\bs r, \bs r')$ have further decomposition
\begin{equation}\label{decomposition2}
u_{\ell\ell'}^{\mathfrak a}(\bs r, \bs r')=\begin{cases}
\displaystyle u_{\ell 0}^{\mathfrak a1}(\bs r, \bs r'),\\
\displaystyle u_{\ell\ell'}^{\mathfrak a1}(\bs r, \bs r')+u_{\ell\ell'}^{\mathfrak a2}(\bs r, \bs r'),\quad 0<\ell<L,\quad \mathfrak a=1,2,\\
\displaystyle u_{\ell L}^{\mathfrak a2}(\bs r, \bs r'),\\
\end{cases}
\end{equation}
while each component has Sommerfeld-type integral representation:
\begin{equation}\label{generalcomponents}
u_{\ell\ell'}^{\mathfrak{ab}}(\bs r, \bs r')=\frac{1}{8\pi^2 }\int_0^{\infty}\int_0^{2\pi}\frac{\lambda_{\rho}}{\lambda_{\ell z}}\mathcal{E}_{\ell\ell'}^{\mathfrak{ab}}(\bs r, \bs r')\sigma_{\ell\ell'}^{\mathfrak{ab}}(\lambda_{\rho})d\alpha d\lambda_{\rho},\quad \mathfrak{a, b}=1, 2.
\end{equation}
Here, $\bs\rho'=(x', y')$ is the source coordinates in $x-y$ plane, $\{\mathcal{E}_{\ell\ell'}^{\mathfrak{ab}}(\bs r, \bs r')\}_{\mathfrak{a, b}=1,2}$ are exponential functions defined as
\begin{equation}\label{zexponential}
\begin{split}
&\mathcal{E}_{\ell\ell'}^{11}(\bs r,\bs r'):=e^{\ri\bs\lambda_{\alpha}\cdot(\bs\rho-\bs\rho')- \lambda_{\ell z} (z-d_{\ell})-\lambda_{\ell'z}(z'-d_{\ell'})},\\
& \mathcal{E}_{\ell\ell'}^{12}(\bs r,\bs r'):=e^{\ri\bs\lambda_{\alpha}\cdot(\bs\rho-\bs\rho')-\lambda_{\ell z} (z-d_{\ell})-\lambda_{\ell'z}(d_{\ell'-1}-z')},\\
&\mathcal{E}_{\ell\ell'}^{21}(\bs r,\bs r'):=e^{\ri\bs\lambda_{\alpha}\cdot(\bs\rho-\bs\rho')-\lambda_{\ell z} (d_{\ell-1}-z)-\lambda_{\ell'z}(z'-d_{\ell'})},\\
& \mathcal{E}_{\ell\ell'}^{22}(\bs r,\bs r'):=e^{\ri\bs\lambda_{\alpha}\cdot(\bs\rho-\bs\rho')-\lambda_{\ell z} (d_{\ell-1}-z)-\lambda_{\ell'z}(d_{\ell'-1}-z')}.
\end{split}
\end{equation}

Since the reaction components of the Green's function in layered media have different expressions \eqref{greenfuncomponent} and \eqref{decomposition2} for source and target particles in different layers, it is necessary to perform calculation individually for
interactions between any two groups of particles among the $L+1$ groups
$\{\mathscr{P}_{\ell}\}_{\ell=0}^{L}$. Applying expressions \eqref{reactioncomponent}, \eqref{greenfuncomponent} and \eqref{decomposition2} in \eqref{potential1}, we obtain
\begin{equation}\label{totalinteraction}
\begin{split}
\Phi_{\ell}(\boldsymbol{r}_{\ell i})=&\Phi_{\ell}^{\text{free}}(\boldsymbol{r}_{\ell 	i})+\Phi_{\ell}^{\text{react}}(\boldsymbol{r}_{\ell i})\\
=&\Phi_{\ell}^{\text{free}}(\boldsymbol{r}_{\ell i})+\sum\limits_{\ell^{\prime}=0}^{L-1}[\Phi_{\ell\ell^{\prime}}^{11
}(\boldsymbol{r}_{\ell i})+\Phi_{\ell\ell^{\prime}}^{21
}(\boldsymbol{r}_{\ell i})]+\sum\limits_{\ell^{\prime}=1}^{L}[\Phi_{\ell\ell^{\prime}}^{12
}(\boldsymbol{r}_{\ell i})+\Phi_{\ell\ell^{\prime}}^{22
}(\boldsymbol{r}_{\ell i})],
\end{split}
\end{equation}
where
\begin{equation}\label{freereactioncomponents}
\begin{split}
&  \Phi_{\ell}^{\text{free}}(\boldsymbol{r}_{\ell i}):=\sum\limits_{j=1,j\neq
	i}^{N_{\ell}}Q_{\ell j}\frac{e^{-\lambda_{\ell}|\bs r_{\ell i}-\bs r_{\ell j}|}}{4\pi\varepsilon_{\ell}|\boldsymbol{r}_{\ell
		i}-\boldsymbol{r}_{\ell j}|},\quad  \Phi_{\ell\ell^{\prime}}^{\mathfrak{ab}}(\boldsymbol{r}_{\ell i}):=\sum
\limits_{j=1}^{N_{\ell^{\prime}}}Q_{\ell^{\prime}j}u_{\ell\ell^{\prime}%
}^{\mathfrak{ab}}(\boldsymbol{r}_{\ell i},\boldsymbol{r}_{\ell^{\prime}j}%
).
\end{split}
\end{equation}
It is clearly that the free space component $\Phi_{\ell}^{\text{free}}(\boldsymbol{r}_{\ell i})$ can be computed using Yukawa-FMM. Therefore, we only focus on the computation of the reaction components $\{\Phi_{\ell\ell^{\prime}}^{\mathfrak{ab}}(\boldsymbol{r}_{\ell i})\}$, $\mathfrak{a, b}=1, 2$ in the $\ell$-th layer.

\section{FMM for 3-D linearized Poisson-Boltzmann equation in layered media}
In this section, we first review the MEs and LEs for the free space Green's function of the linearized PB equation and the corresponding shifting/translation operators. They are the key formulas used in the Yukawa-FMM which we will adopt for the computation of the free space components given in \eqref{freereactioncomponents}. Then, a new derivation for the ME and LE using an integral representation of modified spherical Bessel functions are presented. This new technique will be applied to derive MEs, LEs and M2L translations in the development of the FMMs for the reaction components of the layered media Green's function later on.

\subsection{The multipole and local expansions of free space Green's function}
Define the spherical harmonics
\begin{equation}
Y_n^m(\theta,\phi)=(-1)^m\sqrt{\frac{2n+1}{4\pi}\frac{(n-m)!}{(n+m)!}}P_n^m(\cos\theta)e^{\ri m\phi}=\widehat P_n^m(\cos\theta)e^{\ri m\phi},
\end{equation}
for $n=0, 1, \cdots$, $0 \le |m| \le n$, where $P_n^m(\cos\theta)$ is the associated Legendre function and $\widehat P_n^m(\cos\theta)$ is its normalized version.
Suppose $\bs r_j=(r_j, \theta_j,\varphi_j)$ is the position vector of a point $P$ in spherical coordinates with respect to a given center $O_j$ for $j=1, 2$, and $\bs b=(b, \alpha,\beta)$ is the position vector of $O_1$ with respect to $O_2$, such that $\bs r_2=\bs r_1+\bs b$.  By the relations
\begin{equation}
k_n(z)=-\frac{\pi}{2}\ri^nh_n^{(1)}(\ri z), \quad i_n(z)=\ri^{-n}j_n(iz),
\end{equation}
and the addition theorems of spherical Bessel functions (cf. \cite{clercx1993alternative}), there holds the following addition theorems.
\begin{theorem}\label{zeroorderaddthm}
	Let $\bs r_2=\bs r_1+\bs b$. Then
	\begin{equation}\label{sphbessel00}
	 k_0(\lambda r_2)=4\pi\sum\limits_{n=0}^{\infty}\sum\limits_{m=-n}^n(-1)^nk_n(\lambda b)\overline{Y_n^m(\alpha,\beta)}i_n(\lambda r_1)Y_n^m(\theta_1,\varphi_1)
	\end{equation}
	for $r_1<b$, and
	\begin{equation}\label{sphbessel01}
	 k_0(\lambda r_2)=4\pi\sum\limits_{n=0}^{\infty}\sum\limits_{m=-n}^n(-1)^ni_n(\lambda b)\overline{Y_n^m(\alpha,\beta)}k_n(\lambda r_1)Y_n^m(\theta_1,\varphi_1)
	\end{equation}
	for $r_1>b$.
\end{theorem}
\begin{theorem}\label{additionthmbesselj}
	Let $\bs r_2=\bs r_1+\bs b$. Then
	\begin{equation}
	i_n(\lambda r_2)Y_n^m(\theta_2,\varphi_2)=\sum\limits_{\nu=0}^{\infty}\sum\limits_{\mu=-\nu}^{\nu}\widehat{S}_{n\nu}^{m\mu}(\bs b)i_{\nu}(\lambda r_1)Y_{\nu}^{\mu}(\theta_1,\varphi_1),
	\end{equation}
	where
	\begin{equation}
	\label{spmhat3d}
	\widehat{S}_{n\nu}^{m\mu}(\bs b)=4\pi\sum\limits_{q=0}^{\infty}(-1)^{\nu-n+m+q}i_q(\lambda b)\overline{Y_q^{\mu-m}(\alpha,\beta)}\mathcal{G}(n,m;\nu,-\mu;q),
	\end{equation}
	with $\mathcal{G}(n,m;\nu,-\mu;q)$ being the Gaunt coefficient.
\end{theorem}

\begin{theorem}
	\label{thmhlynm}
	Let $\bs r_2=\bs r_1+\bs b$. Then
	\begin{equation}
	k_n(\lambda r_2)Y_n^m(\theta_2,\varphi_2)=\sum\limits_{\nu=0}^{\infty}\sum\limits_{\mu=-\nu}^{\nu}{S}_{n\nu}^{m\mu}(\bs b)i_{\nu}(\lambda r_1)Y_{\nu}^{\mu}(\theta_1,\varphi_1),
	\end{equation}
	for $r_1<b$, and
	\begin{equation}
	k^{(1)}_n(\lambda r_2)Y_n^m(\theta_2,\varphi_2)=\sum\limits_{\nu=0}^{\infty}\sum\limits_{\mu=-\nu}^{\nu}(-1)^{\nu-n+m}\widehat{S}_{n\nu}^{m\mu}(\bs b)k^{(1)}_{\nu}(\lambda r_1)Y_{\nu}^{\mu}(\theta_1,\varphi_1),
	\end{equation}
	for $r_1>b$, where $\widehat{S}_{n\nu}^{m\mu}(\bs b)$ is given by \eqref{spmhat3d} and
	\begin{equation}\label{singularseparation}
	\begin{split}
	{S}_{n\nu}^{m\mu}(\bs b)
	&=2\pi^2(-1)^{m+\nu+1}\sum\limits_{q=0}^{\infty}k_q(\lambda b)\overline{Y_q^{\mu-m}(\alpha,\beta)}\mathcal{G}(n,m;\nu,-\mu;q),
	\end{split}
	\end{equation}
	and $\mathcal{G}(n,m;\nu,-\mu;q)$ is a Gaunt coefficient.
\end{theorem}

The Gaunt coefficient $\mathcal{G}(n,m;\nu,\mu;q)$ is defined using the Wigner $3-j$ symbol. Although, there are explicit formulas \eqref{spmhat3d} and \eqref{singularseparation} for the separation matrices $\widehat{S}_{n\nu}^{m\mu}(\bs b)$ and ${S}_{n\nu}^{m\mu}(\bs b)$, they are too complicated to be used directly for practical computations. Recurrence formulas (cf. \cite{chew1992recurrence,gumerov2004recursions}) are usually more preferable for their computations.

\begin{figure}[ht!]\label{3dspherical}
	\centering
	\includegraphics[scale=0.9]{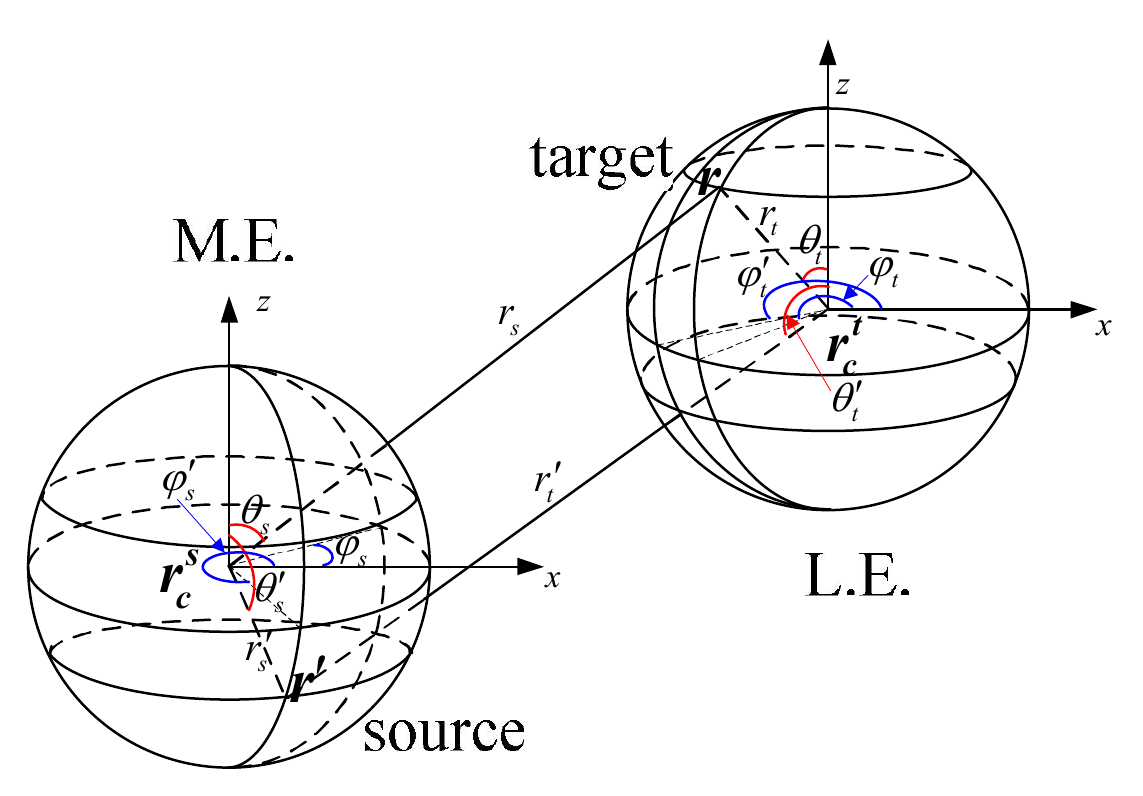}
	\caption{Spherical coordinates used in multipole and local expansions.}
\end{figure}
With these addition theorems, we can present the  multipole and local expansions used in the Yukawa-FMM for Yukawa potential (cf. \cite{greengard2002new, huang2009fmm}). Consider the free space Green's function of the linearized PB equation with a source and a target at $\bs r'$ and $\bs r$, respectively. By using the addition Theorem \ref{zeroorderaddthm}, we have the ME with respect to a source center $\bs r_c^s$:
\begin{equation}\label{freespace3dmulexp}
k_0(\lambda|\bs r-\bs r'|)=\frac{\pi}{2}\frac{e^{-\lambda|\bs r-\bs r'|}}{\lambda|\bs r-\bs r'|}=\sum\limits_{n=0}^{\infty}\sum\limits_{m=-n}^nM_{nm}k_n(\lambda r_s)Y_n^m(\theta_s,\varphi_s),
\end{equation}
and the LE with respect to a target center $\bs r_c^t$:
\begin{equation}\label{freespace3dlocexp}
k_0(\lambda|\bs r-\bs r'|)=\sum\limits_{|n|=0}^{\infty}\sum\limits_{m=-n}^nL_{nm}i_n(\lambda r_t)Y_n^m(\theta_t,\varphi_t),
\end{equation}
where
\begin{equation}\label{melecoeffree}
M_{nm}=4\pi i_n(\lambda r_s')\overline{Y_n^{m}(\theta'_s,\varphi_s')},\quad L_{nm}=4\pi k_n(\lambda r_t')Y_n^{-m}(\theta_t',\varphi_t'),
\end{equation}
$\bs r_c^s$ is the source center close to $\bs r'$ and $\bs r_c^t$ is the target center close to $\bs r$, $(r_s, \theta_s,\varphi_s)$, $(r_t,\theta_t,\varphi_t)$ are the spherical coordinates of $\bs r-\bs r_c^s$ and $\bs r-\bs r_c^t$, $(r'_s, \theta'_s,\varphi'_s)$, $(r'_t,\theta'_t,\varphi'_t)$ are the spherical coordinates of $\bs r'-\bs r_c^s$ and $\bs r'-\bs r_c^t$( see Fig. \ref{3dspherical}).

Applying addition Theorem \ref{thmhlynm} to $k_n^{(1)}(\lambda r_s)Y_n^m(\theta_s, \varphi_s)$ in \eqref{freespace3dmulexp}, the translation from the ME \eqref{freespace3dmulexp} to the LE \eqref{freespace3dlocexp} is given by
\begin{equation}\label{metole}
L_{nm}=\sum\limits_{|\nu|=0}^{\infty}\sum\limits_{\mu=-\nu}^{\nu}{S}_{n\nu}^{m\mu}(\bs r_c^t-\bs r_c^s)M_{\nu\mu}.
\end{equation}
Similarly, we can shift the centers of MEs and LEs via the following translations,
\begin{equation}\label{metome}
\begin{split}
\tilde{M}_{nm}
=&\sum\limits_{\nu=0}^{\infty}\sum\limits_{\mu=-\nu}^{\nu}\overline{\widehat{S}_{n\nu}^{m\mu}(\bs r_c^s-\tilde{\bs r}_c^s)}M_{\nu\mu},\quad \tilde L_{nm}=\sum\limits_{n=0}^{\infty}\sum\limits_{\mu=-\nu}^{\nu} \widehat{S}_{\nu n}^{\mu m}(\tilde{\bs r}_c^t-\bs r_c^t)L_{\nu\mu},
\end{split}
\end{equation}
where
\begin{equation}
\tilde{M}_{nm}=4\pi i_n(\lambda\tilde r_s')\overline{Y_{n}^{m}(\tilde\theta'_s,\tilde\varphi_s')},\quad \tilde L_{nm}=4\pi k_{n}(\lambda\tilde r_t')\overline{Y_{n}^{m}(\tilde\theta_t',\tilde\varphi_t')}
\end{equation}
are the coefficients of the ME and LE with respect to new centers $\tilde{\bs r}_s$ and $\tilde{\bs r}_t$, respectively.

Two important features in \eqref{freespace3dmulexp}-\eqref{freespace3dlocexp} are (i) the source and target coordinates are separated; (ii) they both have exponential convergence. These are the key features for the compression in the Yukawa-FMM (cf. \cite{greengard2002new,huang2009fmm}). Besides using addition theorem, a new approach which can handle Green's function in layered media has been proposed for Helmholtz and Laplace equations in layered media (cf. \cite{wang2019fast, zhang2018exponential,wang2019fastlaplace}).

\subsection{A new derivation for the multipole and local expansions and translation operator}
The Green's function of the linearized PB equation in free space is the modified spherical Bessel function, which has the Sommerfeld-type integral representation
\begin{equation}\label{sommerfeldid}
k_0(\lambda|\bs r|)=\frac{\pi}{2}\frac{e^{-\lambda|\bs r|}}{\lambda|\bs r|}=\frac{1}{4\lambda}\int_0^{\infty}\int_0^{2\pi}\lambda_{\rho}e^{\ri \lambda_{\rho}(x\cos\alpha+y\sin\alpha)}\frac{e^{-\lambda_z|z|}}{\lambda_z}d\alpha  d\lambda_{\rho},
\end{equation}
where $\lambda_z=\sqrt{\lambda^2+\lambda_{\rho}^2}$. In the spectral domain, the source-target separation can be achieved straightforwardly as
\begin{equation}
\label{positivecase}
\begin{split}
k_0(\lambda|\bs r-\bs r'|)=\frac{1}{4\lambda}\int_0^{\infty}\int_0^{2\pi}\lambda_{\rho}\frac{e^{ \bs \lambda\cdot(\bs r-\bs r_c^s)}e^{-\bs \lambda\cdot(\bs r'-\bs r_c^s)}}{\lambda_z}d\alpha  d\lambda_{\rho},\\
k_0(\lambda|\bs r-\bs r'|)=\frac{1}{4\lambda}\int_0^{\infty}\int_0^{2\pi}\lambda_{\rho}\frac{e^{\bs \lambda\cdot(\bs r-\bs r_c^t)}e^{-\bs \lambda\cdot(\bs r'-\bs r^t_c)}}{\lambda_z}d\alpha  d\lambda_{\rho},
\end{split}
\end{equation}
for $z\geq z'$, where $\bs \lambda=(\ri\lambda_{\rho}\cos\alpha, \ri\lambda_{\rho}\sin\alpha,  -\lambda_z).$
Without loss of generality, here we only consider the case $z\geq z'$ for an illustration.

The FMM for Helmholtz equation in layered media use similar source/target separation in its spectral domain. One of the key ingredient is the following extension of the well-known Funk-Hecke formula (cf. \cite{watson,martin2006multiple,wang2019fast}).
\begin{proposition}\label{prop:Funk-Hecke}
	Given $\bs r=(x, y, z)\in \mathbb R^3$, $k>0$, $\alpha\in[0, 2\pi)$ and denoted by $(r,\theta,\varphi)$ the spherical coordinates of $\bs r$, $\bs k=(\sqrt{k^2-k_z^2}\cos\alpha, \sqrt{k^2-k_z^2}\sin\alpha, k_z)$ is a vector of complex entries. Choosing branch \eqref{branch} for $\sqrt{k^2-k_z^2}$ in $e^{\ri \bs k\cdot{\bs r}}$ and $\widehat P_n^m(\frac{k_z}{k})$, then
	\begin{equation}\label{extfunkhecke}
	e^{\ri \bs k\cdot{\bs r}}=\sum\limits_{n=0}^{\infty}\sum\limits_{m=-n}^n A_{n}^m(\bs r)\ri^n\widehat{P}_n^m\Big(\frac{k_z}{k}\Big)e^{-\ri m\alpha}=\sum\limits_{n=0}^{\infty}\sum\limits_{m=-n}^n \overline{A_{n}^m(\bs r)}\ri^n\widehat{P}_n^m\Big(\frac{k_z}{k}\Big)e^{\ri m\alpha},
	\end{equation}
	holds for all $k_z\in\mathbb C$, where
	$$A_{n}^m(\bs r)=4\pi j_n(kr)Y_n^m(\theta,\varphi).$$
\end{proposition}

This extension enlarges the range of the classic Funk-Hecke formula from $k_z\in (-k, k)$ to the whole complex plane by choosing branch
\begin{equation}\label{branch}
\sqrt{k^2-k_z^2}=-\ri \sqrt{r_1r_2}e^{\ri\frac{\theta_1+\theta_2}{2}},
\end{equation}
in the square root function $\sqrt{k^2-k_z^2}$.
Here $(r_i,\theta_i), i=1, 2$ are the modules and principle values of the arguments of complex numbers $k_z+k$ and $k_z-k$, i.e.,
$$k_z+k=r_1e^{\ri\theta_1}, \quad -\pi<\theta_1\leq\pi,\quad  k_z-k=r_2e^{\ri\theta_2},\quad -\pi<\theta_2\leq\pi.$$
There, it is enough to consider the case $k>0$ is a real positive number. Note that the linarized PB equation can be obtained from Helmholtz equation via modification $k\rightarrow\ri\lambda$. Therefore, we shall prove another extension of the Funk-Hecke formula to allow pure imaginary $k=\ri\lambda$.

By using the branch defined in \eqref{branch} for the square roots, we have the extension of the well-known Legendre addition theorem \cite[p.395]{Whi.W27}.
\begin{lemma}\label{lemma2}
	Let $\bs w=(\sqrt{1-w^2}\cos\alpha, \sqrt{1-w^2}\sin\alpha, w)$ be a vector with complex entries, $\theta, \phi$ be the azimuthal angle and polar angles of a unit vector $\hat{\bs r}$. Define
	\begin{equation}
	\beta(w)=w\cos\theta+\sqrt{1-w^2}\sin\theta\cos(\alpha-\phi),
	\end{equation}
	then
	\begin{equation}\label{legendreadd}
	P_n(\beta(w))=\frac{4\pi}{2n+1}\sum\limits_{m=-n}^n\widehat P_n^m(\cos\theta)\widehat P_n^m(w)e^{\ri m(\alpha-\phi)},
	\end{equation}
	for all $w\in\mathbb C$.
\end{lemma}
The following Lemma is actually the same conclusion of Lemma 4 in \cite{wang2019fast}. Here we make it more general by enlarge the domain of $a$ to be any complex number.
\begin{lemma}\label{lemma1}
	For any complex number $a$, there holds
	\begin{equation}\label{planewaveexp}
	e^{a z}=\sum\limits_{n=0}^{\infty}(2n+1)i_n(a)P_n(z), \quad\forall z\in\mathbb C,
	\end{equation}
	where $i_n(a)=\sqrt{\frac{\pi}{2a}}I_{n+1/2}(a)$ is the modified  spherical Bessel function of the first kind, $P_n(z)$ is the Legendre polynomial extended to the complex plane.
\end{lemma}
\begin{proof}
	Recall the series (cf. \cite[10.60.8]{Olver2010})
	\begin{equation}
	e^{a\cos\theta}=\sum\limits_{n=0}^{\infty}(2n+1)i_n(a)P_n(\cos\theta),
	\end{equation}
	we can see that \eqref{planewaveexp} holds for all $z\in [-1, 1]$. Next, we consider its extension to the whole complex plane. Apparently, $e^{a z}$ is an entire function of $z$. Meanwhile, the spherical Bessel function $i_n(a)$ has the following upper bound (cf. \cite[9.1.62]{Abr.I64})
	\begin{equation}
	|i_n(a)|=|\ri^{-n}j_n(\ri a)|\leq \frac{\Gamma(\frac{3}{2})}{\Gamma(n+\frac{3}{2})}\Big(\frac{a}{2}\Big)^n\leq \frac{1}{n!}\Big(\frac{a}{2}\Big)^n.
	\end{equation}
	Obviously, the extension of the Legendre polynomial $P_n(z)$ to the whole complex plane is a polynomial of degree $n$ with $n$ distinct roots $\{z_j\}_{j=1}^n$ in the interval $[-1, 1]$. Therefore,
	\begin{equation}
	|P_n(z)|=|a_n|\prod\limits_{j=1}^n|z-z_j|\leq 2^n(|z|+1)^n, \quad \forall z\in\mathbb C,
	\end{equation}
	here the estimate $a_n=\frac{(2n)!}{2^n(n!)^2}\leq 2^n$ for the coefficient of the leading term of $P_n(z)$ is used. These upper bounds for $i_n(a)$ and $P_n(z)$ give an estimate
	\begin{equation}
	\sum\limits_{n=0}^{\infty}(2n+1)|i_n(a)P_n(z)|\leq \sum\limits_{n=0}^{\infty}(2n+1)\frac{a^n(|z|+1)^n}{n!}=(2a(|z|+1)+1)e^{a(|z|+1)}.
	\end{equation}
	It is easy to show that the series on the righthand side of \eqref{planewaveexp} converges uniformly in any compact set $D\subset \mathbb C$ and hence converges to an entire function of $z$. By the analytic extension theory, we complete the proof.
\end{proof}

\begin{proposition}\label{prop:modiffied-Funk-Hecke}
	Given $\bs r=(x, y, z)\in \mathbb R^3$, $\lambda >0$, $\alpha\in[0, 2\pi)$ and denoted by $(r,\theta,\varphi)$ the spherical coordinates of $\bs r$, $\bs \lambda=(\ri\lambda_{\rho}\cos\alpha, \ri\lambda_{\rho}\sin\alpha, -\sqrt{\lambda^2+\lambda_{\rho}^2})$ is a vector of complex entries. Choosing the branch \eqref{branch} for $\sqrt{\lambda_z^2-\lambda^2}$ in $e^{\ri \bs \lambda\cdot{\bs r}}$ and $\widehat P_n^m(\frac{\lambda_z}{\lambda})$, then
	\begin{equation}\label{extmodifiedfunkhecke}
	e^{\bs \lambda\cdot{\bs r}}=\sum\limits_{n=0}^{\infty}\sum\limits_{m=-n}^n \overline{B_{n}^m(\bs r)}\widehat{P}_n^m\Big(\frac{\lambda_z}{\lambda}\Big)e^{\ri m\alpha}=\sum\limits_{n=0}^{\infty}\sum\limits_{m=-n}^n B_{n}^m(\bs r)\widehat{P}_n^m\Big(\frac{\lambda_z}{\lambda}\Big)e^{-\ri m\alpha},
	\end{equation}
	holds for all $\lambda_{\rho}\in\mathbb C$, where
	$$B_{n}^m(\bs r)=4\pi (-1)^ni_n(\lambda r)Y_n^m(\theta,\varphi).$$
\end{proposition}
\begin{proof}
	Define
	$$\beta=-\sqrt{1+\frac{\lambda_{\rho}^2}{\lambda^2}}\cos\theta+\ri\frac{\lambda_{\rho}}{\lambda}\sin\theta\cos(\alpha-\phi),$$
	then, $\lambda r\beta=\bs \lambda\cdot\bs r$. Let $a=\lambda r$,  $z=\beta$ in \eqref{planewaveexp}, we have
	\begin{equation}
	e^{\bs \lambda\cdot\bs r}=\sum\limits_{n=0}^{\infty}(2n+1)i_n(\lambda r)P_n(\beta).
	\end{equation}
	Then, the spherical harmonic expansion \eqref{extmodifiedfunkhecke} follows by applying Lemma \ref{lemma2} together with the property $\widehat P_n^m(-z)=(-1)^n\widehat P_n^m(z)$ for all $z\notin [-1, 1]$.
\end{proof}

Applying the spherical harmonic expansion \eqref{extmodifiedfunkhecke} to exponential functions $e^{-\bs \lambda\cdot(\bs r'-\bs r_c^s)}$ and $e^{-\bs \lambda\cdot(\bs r-\bs r_c^t)}$ in \eqref{positivecase} gives
\begin{equation}\label{meinspectraldomain}
k_0(\lambda|\bs r-\bs r'|)=\sum\limits_{n=0}^{\infty}\sum\limits_{m=-n}^{n}\frac{ M_{nm}}{4\lambda}\int_0^{\infty}\int_0^{2\pi}\lambda_{\rho}\frac{e^{\bs \lambda\cdot(\bs r-\bs r_c^s)}}{ \lambda_z}\widehat{P}_n^m\Big(\frac{\lambda_z}{\lambda}\Big)e^{\ri m\alpha}d\alpha  d\lambda_{\rho},
\end{equation}
and
\begin{equation}\label{leinspectraldomain}
k_0(\lambda|\bs r-\bs r'|)=\sum\limits_{n=0}^{\infty}\sum\limits_{m=-n}^{n}\hat L_{nm}i_n(kr_t)Y_n^m(\theta_t,\phi_t),
\end{equation}
for $z\geq z'$, where $M_{nm}$ is defined in \eqref{melecoeffree} and
\begin{equation}
\hat L_{nm}=\frac{(-1)^n}{4\lambda}\int_0^{\infty}\int_0^{2\pi}\lambda_{\rho}\frac{e^{\bs \lambda\cdot(\bs r_c^t-\bs r')}}{ \lambda_z}\widehat{P}_n^m\Big(\frac{\lambda_z}{\lambda}\Big)e^{-\ri m\alpha}d\alpha  d\lambda_{\rho}.
\end{equation}
{For the convergence of the Sommerfeld-type integrals in the above expansions, we only consider centers such that $z_c^s<z$ and $z_c^t>z'$. }
Recall the identity
\begin{equation}\label{wavefunspectralform}
k_n(\lambda|\bs r|)Y_n^m(\theta,\varphi)=\frac{1}{4\lambda}\int_0^{\infty}\int_0^{2\pi}\lambda_{\rho}\frac{e^{\bs \lambda\cdot\bs r}}{ \lambda_z}\widehat P_n^{m}\Big(\frac{\lambda_z}{\lambda}\Big)e^{\ri m\alpha} d\alpha d\lambda_{\rho}
\end{equation}
for $z\geq 0$, we see that \eqref{meinspectraldomain} and \eqref{leinspectraldomain} are exactly the multipole and local expansions \eqref{freespace3dmulexp}-\eqref{freespace3dlocexp} for the case of $z\geq z'$.

To derive the translation from the multipole expansion \eqref{meinspectraldomain} to the local expansion \eqref{leinspectraldomain}, we perform a further splitting in \eqref{meinspectraldomain}
\begin{equation}
e^{\bs \lambda\cdot(\bs r-\bs r_c^s)}=e^{\bs \lambda\cdot(\bs r-\bs r_c^t)}e^{\bs \lambda\cdot(\bs r_c^t-\bs r_c^s)}
\end{equation}
 and apply expansion \eqref{extmodifiedfunkhecke} to obtain the following translation
\begin{equation*}
\begin{split}
L_{nm}=&\frac{(-1)^{n}}{ 2\lambda}\sum\limits_{\nu=0}^{\infty}\sum\limits_{\mu=-\nu}^{\nu}M_{\nu\mu}\int_0^{\infty}\int_0^{2\pi}\lambda_{\rho}\frac{e^{\bs \lambda\cdot(\bs r_c^t-\bs r_c^s)}}{\lambda_z} \widehat P_{n}^{m}\Big(\frac{\lambda_z}{\lambda}\Big)\widehat{P}_{\nu}^{\mu}\Big(\frac{\lambda_z}{\lambda}\Big)e^{\ri(\mu-m)\alpha}d\alpha  d\lambda_{\rho},
\end{split}
\end{equation*}
which implies an integral representation of $S_{n\nu}^{m\mu}(\bs r_c^t-\bs r_c^s)$ in \eqref{metole}. {In order to ensure the convergence of the Sommerfeld-type integral in the translation operator, the centers are also assumed to satisfy $z_c^t>z_c^s$. }

\subsection{Equivalent polarization sources for reaction components}
Note that free space components only involve interactions between sources in the same layer. All interactions between sources in different layers are included in the reaction components. Two groups of sources involved in the computation of a reaction component could be physically very far away from each other as there could be many layers between the source and target layers associated to the reaction component, see Fig. \ref{polarizedsource} (left).

Our recent work on the Helmholtz equation \cite{zhang2018exponential,wang2019fast} has shown that the exponential convergence of the ME and LE for the reaction components $u_{\ell\ell'}^{\mathfrak{ab}}(\bs r, \bs r')$ in fact depends on the distance between the target and a polarization source defined for the source at $\bs r'$, which uses the distance between the source $\bs r'$ and the nearest material interface and always locates next to the nearest interface adjacent to the target. Fig. \ref{sourceimages} illustrates the location of the polarization charge $\bs r'_{\mathfrak{ab}}$ for each of the four types of reaction fields
$\tilde u_{\ell\ell'}^{\mathfrak{ab}},\mathfrak{a},\mathfrak{b}=1,2 $. Specifically, the equivalent polarization sources associated to reaction components $u^{\mathfrak{ab}}_{\ell\ell'}(\bs r, \bs r')$, $\mathfrak{a}, \mathfrak{b} =1, 2$ are set to be  at coordinates (see Fig. \ref{sourceimages})
\begin{equation}\label{eqpolarizedsource}
\begin{split}
&\bs r'_{11}:=(x', y', d_{\ell}-(z'-d_{\ell'})),\quad\quad\bs r'_{12}:=(x', y', d_{\ell}-(d_{\ell'-1}-z')),\\
&\bs r'_{21}:=(x', y', d_{\ell-1}+(z'-d_{\ell'})),\quad \bs r'_{22}:=(x', y', d_{\ell-1}+(d_{\ell'-1}-z')),
\end{split}
\end{equation}
and the reaction potentials are
\begin{equation}\label{generalcomponentsimag}
\begin{split}
\tilde u_{\ell\ell'}^{1\mathfrak b}(\bs r, \bs r'_{1\mathfrak b}):=\frac{1}{8\pi^2 }\int_0^{\infty}\int_0^{2\pi}\frac{\lambda_{\rho}}{\lambda_{\ell z}}{\mathcal E}^{+}_{\ell\ell'}(\bs r, \bs r'_{1\mathfrak b})\sigma_{\ell\ell'}^{1\mathfrak b}(\lambda_{\rho})d\alpha d\lambda_{\rho},\\
\tilde u_{\ell\ell'}^{2\mathfrak b}(\bs r, \bs r'_{2\mathfrak b}):=\frac{1}{8\pi^2 }\int_0^{\infty}\int_0^{2\pi}\frac{\lambda_{\rho}}{\lambda_{\ell z}}\mathcal E^{-}_{\ell\ell'}(\bs r, \bs r'_{2\mathfrak b})\sigma_{\ell\ell'}^{2\mathfrak b}(\lambda_{\rho})d\alpha d\lambda_{\rho}
\end{split}
\end{equation}
where
\begin{equation}\label{mekernelimage}
\begin{split}
&{\mathcal E}^{+}_{\ell\ell'}(\bs r, \bs r'_{1\mathfrak{b}}):=e^{\ri\bs\lambda_{\alpha}\cdot(\bs\rho-\bs\rho'_{1\mathfrak{b}})}e^{-\lambda_{\ell z}(z-d_{\ell})-\lambda_{\ell'z}(d_{\ell}-z^{\prime}_{1\mathfrak{b}})},\\
&{\mathcal E}^{-}_{\ell\ell'}(\bs r, \bs r'_{2\mathfrak{b}}):=e^{\ri\bs\lambda_{\alpha}\cdot(\bs\rho-\bs\rho'_{2\mathfrak{b}})}e^{-\lambda_{\ell z}(d_{\ell-1}-z)-\lambda_{\ell'z}(z^{\prime}_{2\mathfrak{b}}-d_{\ell-1})},
\end{split}
\end{equation}
and $\bs \rho_{\mathfrak{ab}}'=(x'_{\mathfrak{ab}}, y'_{\mathfrak{ab}})$, $z'_{\mathfrak{ab}}$ denote the $x-y$ and  $z$-coordinate of $\bs r'_{\mathfrak{ab}}$, i.e.,
\begin{equation*}
\begin{split}
&z_{11}^{\prime}=d_{\ell}-(z'-d_{\ell'}),\quad z_{12}^{\prime}=d_{\ell}-(d_{\ell'-1}-z'),\\
&z_{21}^{\prime}=d_{\ell-1}+(z'-d_{\ell'}),\quad z_{22}^{\prime}=d_{\ell-1}+(d_{\ell'-1}-z').
\end{split}
\end{equation*}

We can see that the reaction potentials \eqref{generalcomponentsimag} represented by using the equivalent polarization sources has similar form as the Sommerfeld-type integral representation \eqref{generalcomponents}.
\begin{figure}[ht!]
	\center
	\subfigure[$u_{\ell\ell'}^{11}$]{\includegraphics[scale=0.45]{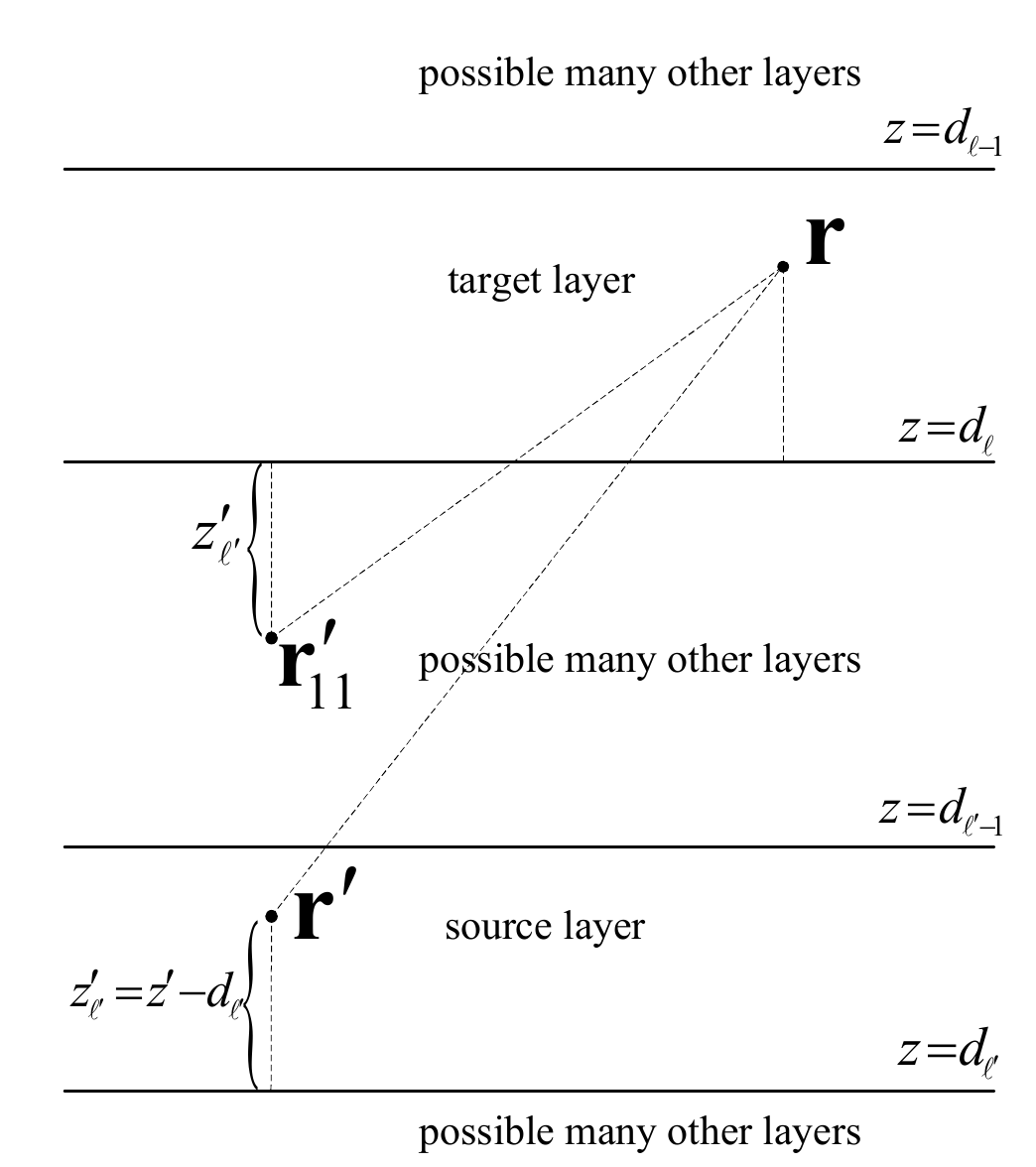}}
	\subfigure[$u_{\ell\ell'}^{12}$]{\includegraphics[scale=0.45]{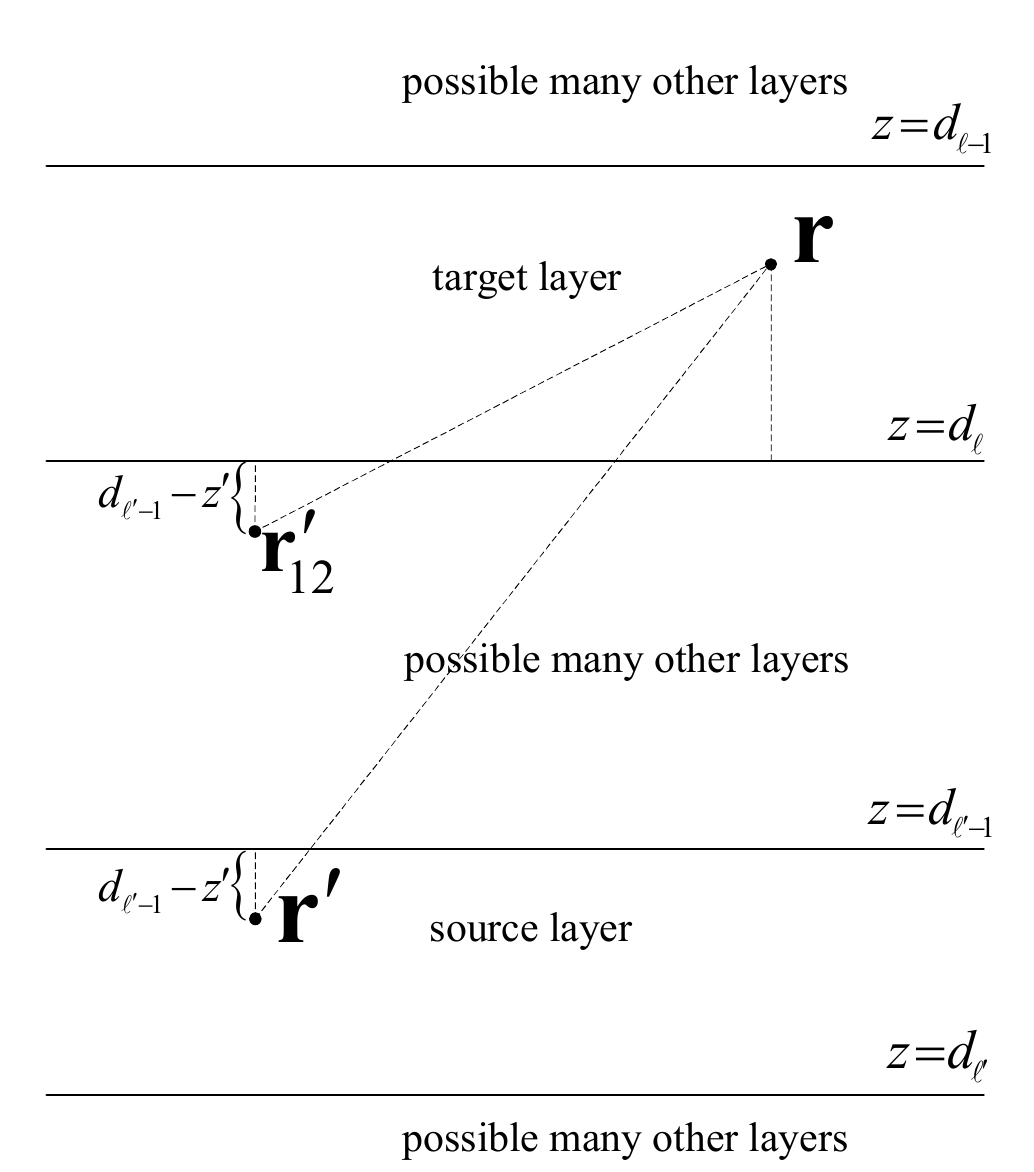}} \\
	
	\subfigure[$u_{\ell\ell'}^{21}$]{\includegraphics[scale=0.45]{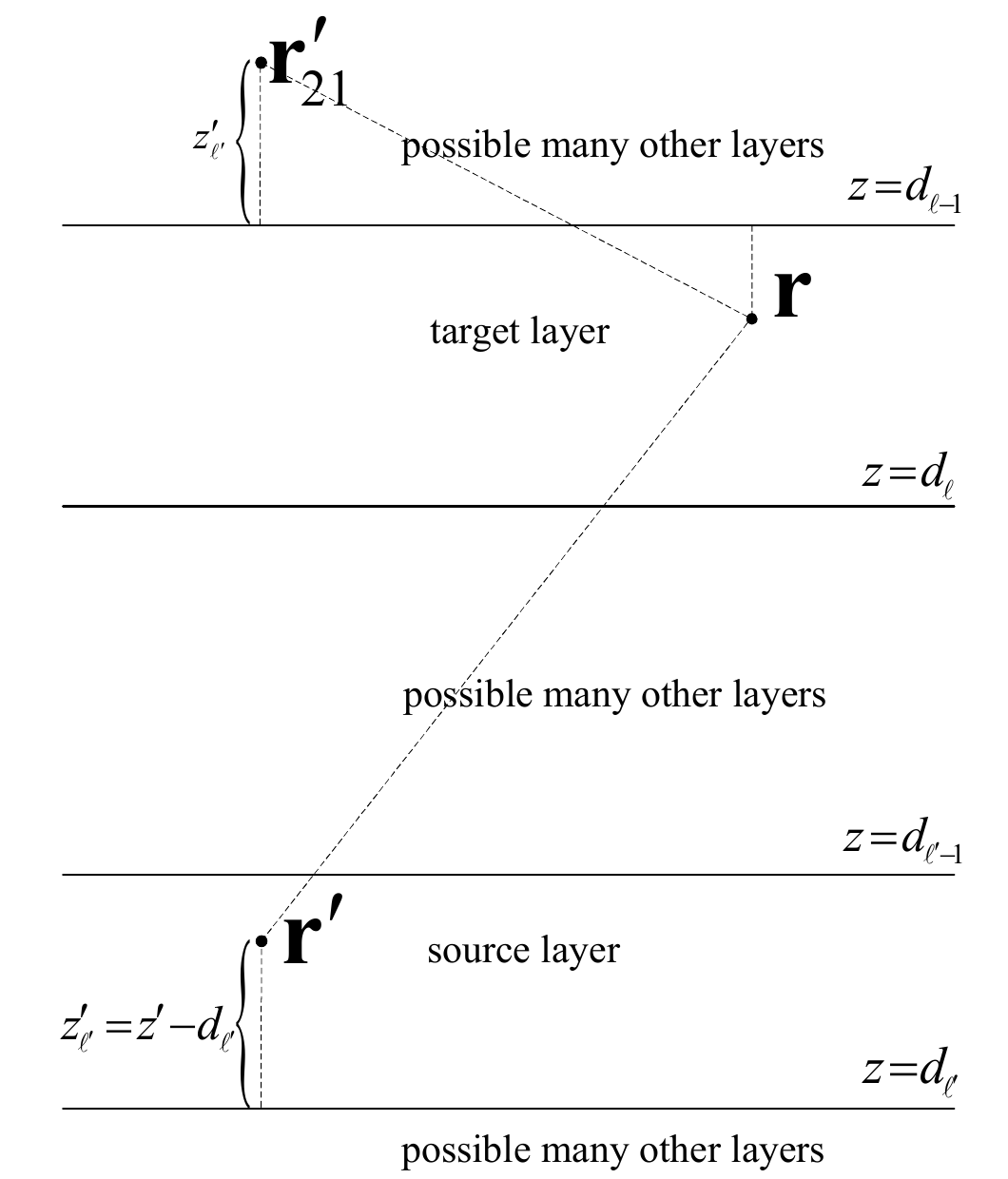}}
	\subfigure[$u_{\ell\ell'}^{22}$]{\includegraphics[scale=0.45]{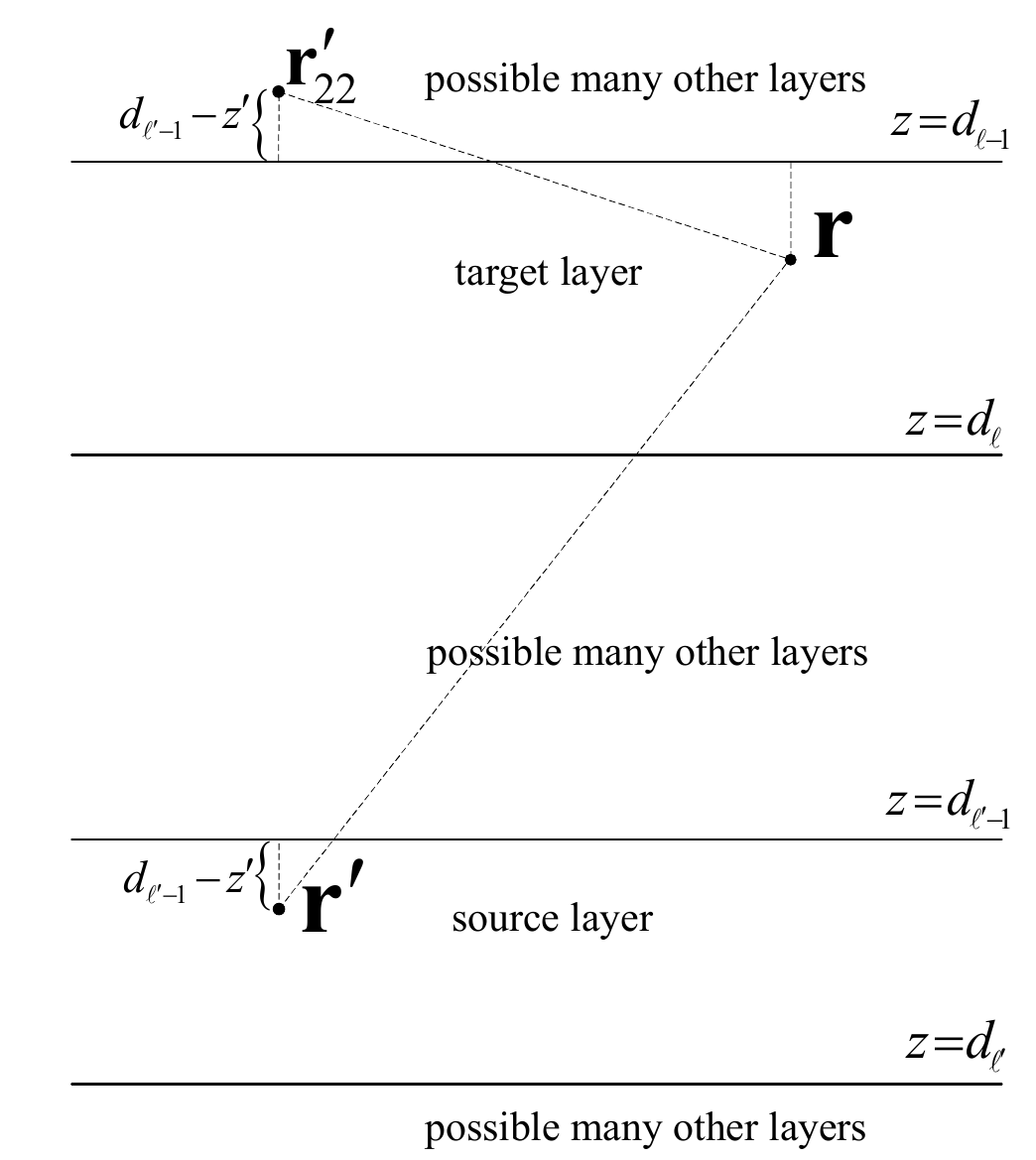}}
	\caption{Location of equivalent polarization sources for the computation of $u_{\ell\ell'}^{\mathfrak{ab}}$.}%
	\label{sourceimages}%
\end{figure}
Recall the expressions \eqref{zexponential}, one can verify that
\begin{equation}\label{expkernelexp}
\mathcal E_{\ell\ell'}^{1\mathfrak b}(\bs r, \bs r')={\mathcal E}^{+}(\bs r, \bs r'_{1\mathfrak b}),\quad \mathcal E_{\ell\ell'}^{2\mathfrak b}(\bs r, \bs r')={\mathcal E}^{-}(\bs r, \bs r'_{2\mathfrak b}).
\end{equation}
Therefore, the reaction components of layered Green's function can be re-expressed using equivalent polarization coordinates as follows
\begin{equation}
u_{\ell\ell'}^{1\mathfrak b}(\bs r, \bs r')=\tilde u_{\ell\ell'}^{1\mathfrak b}(\bs r, \bs r'_{1\mathfrak b}),\quad
u_{\ell\ell'}^{2\mathfrak b}(\bs r, \bs r')=\tilde u_{\ell\ell'}^{2\mathfrak b}(\bs r, \bs r'_{2\mathfrak b}),\quad \mathfrak b=1, 2.
\end{equation}
Substituting into the expression of $\Phi_{\ell\ell^{\prime}}^{\mathfrak{ab}}(\boldsymbol{r}_{\ell i})$ in \eqref{freereactioncomponents}, we obtain
\begin{equation}\label{reactcompusingpolar}
\Phi_{\ell\ell^{\prime}}^{\mathfrak{ab}}(\boldsymbol{r}_{\ell i}):=\sum
\limits_{j=1}^{N_{\ell^{\prime}}}Q_{\ell^{\prime}j}\tilde u_{\ell\ell^{\prime}%
}^{\mathfrak{ab}}(\boldsymbol{r}_{\ell i},\boldsymbol{r}^{\mathfrak ab}_{\ell^{\prime}j}%
),
\end{equation}
where
\begin{equation}\label{equivpolarcoord}
\begin{split}
&\boldsymbol{r}^{11}_{\ell^{\prime}j}=(x_{\ell'j}, y_{\ell'j}, d_{\ell}-(z_{\ell j}-d_{\ell'})), \quad\;\;\;\, \boldsymbol{r}^{12}_{\ell^{\prime}j}=(x_{\ell'j}, y_{\ell'j}, d_{\ell}-(d_{\ell'-1}-z_{\ell j})),\\
&\boldsymbol{r}^{21}_{\ell^{\prime}j}=(x_{\ell'j}, y_{\ell'j}, d_{\ell-1}+(z_{\ell j}-d_{\ell'})), \quad \boldsymbol{r}^{22}_{\ell^{\prime}j}=(x_{\ell'j}, y_{\ell'j}, d_{\ell-1}+(d_{\ell'-1}-z_{\ell j})),
\end{split}
\end{equation}
are equivalent polarization coordinates of $\bs r_{\ell' j}$ for the computation of reaction components in the $\ell$-th layer, see Fig \ref{polarizedsource} for an illustration of $\{\bs r_{\ell'j}^{11}\}_{j=1}^{N_{\ell'}}$ and $\{\bs r_{\ell'j}^{21}\}_{j=1}^{N_{\ell'}}$.
\begin{figure}[ht!]
	\centering
	\includegraphics[scale=0.65]{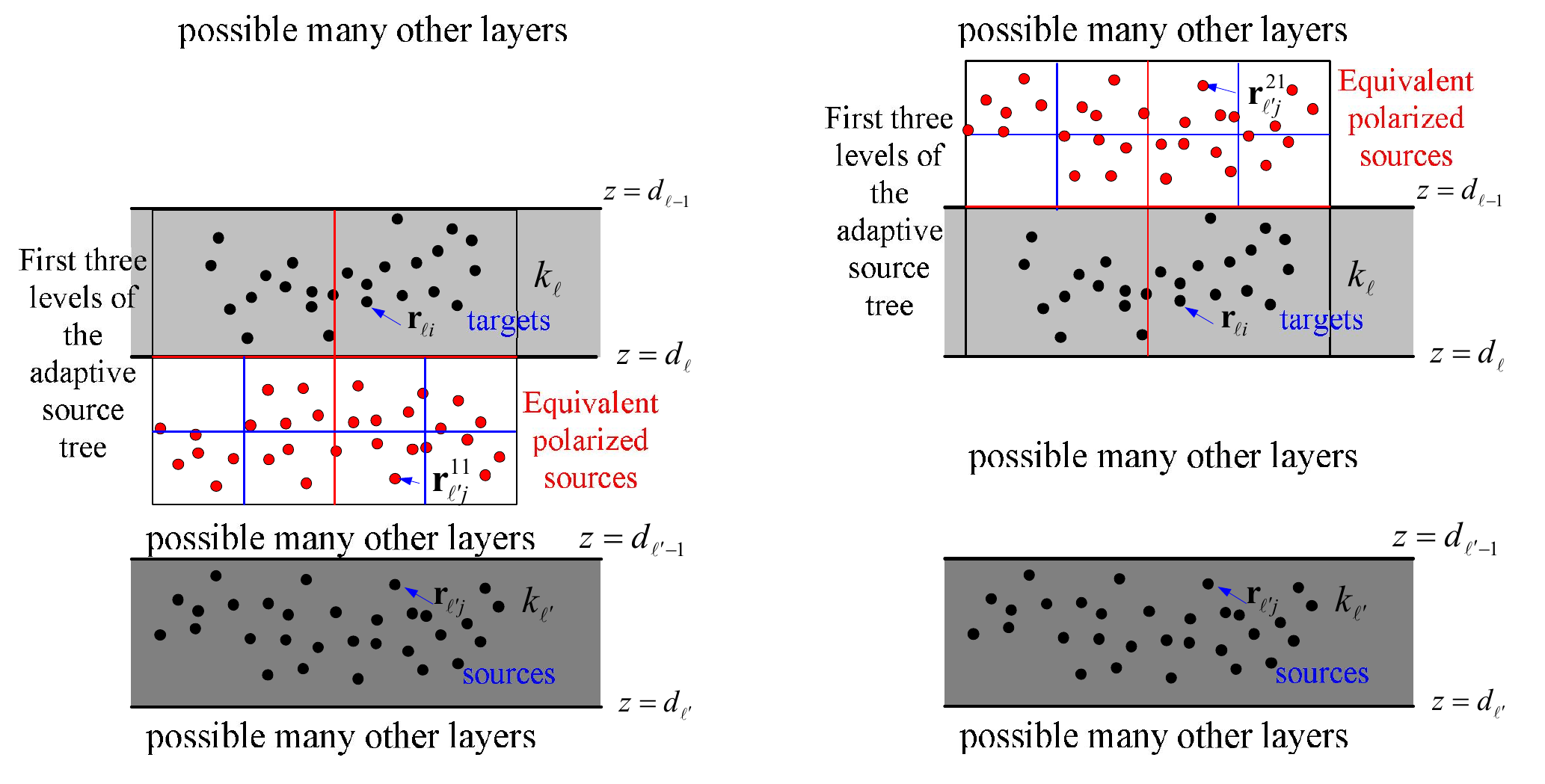}
	\caption{Equivalent polarized sources $\{\bs r_{\ell'j}^{11}\}$, $\{\bs r_{\ell'j}^{21}\}$ and boxes in source tree.}
	\label{polarizedsource}
\end{figure}

By using the expression \eqref{reactcompusingpolar}, the computation of the reaction components can be performed between targets and associated equivalent polarization sources. The definition given by \eqref{equivpolarcoord} shows that the target particles $\{\bs r_{\ell i}\}_{i=1}^{N_{\ell}}$ and the corresponding equivalent polarization sources  are always located on different sides of an interface $z=d_{\ell-1}$ or $z=d_{\ell}$, see Fig. \ref{polarizedsource}. We still emphasize that the introduced equivalent polarization sources are separate with the targets even in considering the reaction components for sources and targets in the same layer, see the numerical examples given in section 3.4. This property implies significant advantage of introducing equivalent polarization sources and using expression \eqref{reactcompusingpolar} in the FMMs for the reaction components $\Phi_{\ell\ell^{\prime}}^{\mathfrak{ab}}(\boldsymbol{r}_{\ell i})$, $\mathfrak a,b=1,2$. More details about this advantage will be discussed in Remark \ref{remark1}. The numerical results presented in Section 4 also validate that the FMMs for reaction components have high efficiency as a direct consequence of the separation of the targets and equivalent polarization sources by interface.

\subsection{Multipole and local expansions and translation operators for a general reaction component}

In the development of FMM for reaction components $\Phi_{\ell\ell'}^{\mathfrak{ab}}(\bs r_{\ell i})$, we will use expression \eqref{reactcompusingpolar} with equivalent polarization coordinates. Therefore, MEs, LEs and corresponding translation operators for $\tilde u_{\ell\ell'}^{\mathfrak{ab}}(\bs r, \bs r'_{\mathfrak{ab}})$ are required. Recall source/target separation in \eqref{positivecase}, similar separations
\begin{equation}\label{sourcetargetseparationsc}
\begin{split}
{\mathcal E}^{+}_{\ell\ell'}(\bs r, \bs r'_{1\mathfrak b})=\mathcal E_{\ell\ell'}^{+}(\bs r, \bs r^{1\mathfrak b}_c)e^{\ri\bs\lambda_{\alpha}\cdot(\bs\rho_c^{1\mathfrak b}-\bs\rho'_{1\mathfrak b})-\lambda_{\ell'z}(z^{1\mathfrak b}_c-z'_{1\mathfrak b})}\\
{\mathcal E}^{-}_{\ell\ell'}(\bs r, \bs r'_{2\mathfrak b})=\mathcal E_{\ell\ell'}^{-}(\bs r, \bs r^{2\mathfrak b}_c)e^{\ri\bs\lambda_{\alpha}\cdot(\bs\rho_c^{2\mathfrak b}-\bs\rho'_{2\mathfrak b})+\lambda_{\ell'z}(z^{2\mathfrak b}_c-z'_{2\mathfrak b})}
\end{split}
\end{equation}
and
\begin{equation}\label{sourcetargetseparationtc}
\begin{split}
\mathcal E_{\ell\ell'}^+(\bs r, \bs r'_{1\mathfrak b})&=\mathcal E_{\ell\ell'}^{+}(\bs r_c^t, \bs r'_{1\mathfrak b})e^{\ri\bs\lambda_{\alpha}\cdot(\bs\rho-\bs\rho_c^t)-\lambda_{\ell z}(z-z_c)},\\
\mathcal E_{\ell\ell'}^-(\bs r, \bs r'_{2\mathfrak b})&=\mathcal E_{\ell\ell'}^{-}(\bs r_c^t, \bs r'_{2\mathfrak b})e^{\ri\bs\lambda_{\alpha}\cdot(\bs\rho-\bs\rho_c^t)+\lambda_{\ell z}(z-z_c)}.
\end{split}
\end{equation}
can be obtained for $\mathfrak b=1, 2$ by inserting the polarization source centers $\bs r^{\mathfrak {ab}}_c=(x_c^{\mathfrak{ab}},y_c^{\mathfrak{ab}}, z_c^{\mathfrak{ab}} )$ and the target center $\bs r_c^t=(x_c^t, y_c^t, z_c^t)$, respectively. Here we also use notations $\bs\rho_c^{\mathfrak{ab}}=(x_c^{\mathfrak{ab}}, y_c^{\mathfrak{ab}})$, $\bs\rho_c^t=(x_c^t, y_c^t)$.
Moreover, Proposition \ref{prop:modiffied-Funk-Hecke} gives spherical harmonic expansions:
\begin{equation}\label{planewavesphexp}
\begin{split}
&e^{\ri\bs\lambda_{\alpha}\cdot(\bs\rho_c^{1\mathfrak b}-\bs\rho'_{1\mathfrak b})-\lambda_{\ell' z}(z^{1\mathfrak b}_c-z'_{1\mathfrak b})}=\sum\limits_{n=0}^{\infty}\sum\limits_{m=-n}^{n} 4\pi(-1)^n i_n(\lambda_{\ell'}r_s^{1\mathfrak b})\overline{Y_n^m(\pi-\theta^{1\mathfrak b}_s,\pi+\varphi^{1\mathfrak b}_s)}\widehat{P}_n^m\Big(\frac{\lambda_{\ell'z}}{\lambda_{\ell'}}\Big)e^{\ri m\alpha},\\
&e^{\ri\bs\lambda_{\alpha}\cdot(\bs\rho_c^{2\mathfrak b}-\bs\rho'_{2\mathfrak b})+\lambda_{\ell'z}(z^{2\mathfrak b}_c-z'_{2\mathfrak b})}=\sum\limits_{n=0}^{\infty}\sum\limits_{m=-n}^{n} 4\pi(-1)^n i_n(\lambda_{\ell'}r_s^{2\mathfrak b})\overline{Y_n^m(\theta^{2\mathfrak b}_s,\pi+\varphi^{2\mathfrak b}_s)}\widehat{P}_n^m\Big(\frac{\lambda_{\ell'z}}{\lambda_{\ell'}}\Big)e^{\ri m\alpha},\\
&e^{\ri\bs\lambda_{\alpha}\cdot(\bs\rho-\bs\rho_c^t)- \lambda_{\ell z}(z-z_c^t)}=\sum\limits_{n=0}^{\infty}\sum\limits_{m=-n}^{n} 4\pi(-1)^n i_n(\lambda_{\ell}r_t)Y_n^m(\theta_t,\varphi_t)\widehat{P}_n^m\Big(\frac{\lambda_{\ell z}}{\lambda_{\ell}}\Big)e^{-\ri m\alpha},\\
&e^{\ri\bs\lambda_{\alpha}\cdot(\bs\rho-\bs\rho_c^t)+ \lambda_{\ell z}(z-z_c^t)}=\sum\limits_{n=0}^{\infty}\sum\limits_{m=-n}^{n} 4\pi(-1)^n i_n(\lambda_{\ell}r_t)Y_n^m(\pi-\theta_t,\varphi_t)\widehat{P}_n^m\Big(\frac{\lambda_{\ell z}}{\lambda_{\ell}}\Big)e^{-\ri m\alpha},
\end{split}
\end{equation}
where $(r_s^{\mathfrak{ab}}, \theta_s^{\mathfrak{ab}}, \varphi_s^{\mathfrak{ab}})$ is the spherical coordinates of $\bs r'_{\mathfrak{ab}}-\bs r_c^{\mathfrak{ab}}$.
Since
$$Y_n^{m}(\pi-\theta,\varphi)=(-1)^{n+m}Y_n^{m}(\theta,\varphi),\quad Y_n^{m}(\theta,\pi+\varphi)=(-1)^{m}Y_n^{m}(\theta,\varphi),$$
the above spherical harmonic expansions together with source/target separation \eqref{sourcetargetseparationsc} and \eqref{sourcetargetseparationtc} implies
\begin{equation}\label{integrandme}
\begin{split}
\mathcal E_{\ell\ell'}^+(\bs r, \bs r'_{1\mathfrak b})&=\mathcal E_{\ell\ell'}^+(\bs r, \bs r_c^{1\mathfrak b})\sum\limits_{n=0}^{\infty}\sum\limits_{m=-n}^{n} 4\pi i_n(\lambda_{\ell'}r_s^{1\mathfrak b})\overline{Y_n^m(\theta^{1\mathfrak b}_s,\varphi^{1\mathfrak b}_s)}\widehat{P}_n^m\Big(\frac{\lambda_{\ell'z}}{\lambda_{\ell'}}\Big)e^{\ri m\alpha},\\
\mathcal E_{\ell\ell'}^-(\bs r, \bs r'_{2\mathfrak b})&=\mathcal E_{\ell\ell'}^-(\bs r, \bs r_c^{2\mathfrak b})\sum\limits_{n=0}^{\infty}\sum\limits_{m=-n}^{n} 4\pi(-1)^{n+m} i_n(\lambda_{\ell'}r_s^{1\mathfrak b})\overline{Y_n^m(\theta^{1\mathfrak b}_s,\varphi^{1\mathfrak b}_s)}\widehat{P}_n^m\Big(\frac{\lambda_{\ell'z}}{\lambda_{\ell'}}\Big)e^{\ri m\alpha}
\end{split}
\end{equation}
and
\begin{equation}\label{integrandle}
\begin{split}
\mathcal E_{\ell\ell'}^+(\bs r, \bs r'_{1\mathfrak b})&=\mathcal E_{\ell\ell'}^+(\bs r_c^t, \bs r'_{1\mathfrak b})\sum\limits_{n=0}^{\infty}\sum\limits_{m=-n}^{n} 4\pi(-1)^n i_n(\lambda_{\ell}r_t)Y_n^m(\theta_t,\varphi_t)\widehat{P}_n^m\Big(\frac{\lambda_{\ell z}}{\lambda_{\ell}}\Big)e^{-\ri m\alpha},\\
\mathcal E_{\ell\ell'}^-(\bs r, \bs r'_{2\mathfrak b})&=\mathcal E_{\ell\ell'}^-(\bs r_c^t, \bs r'_{2\mathfrak b})\sum\limits_{n=0}^{\infty}\sum\limits_{m=-n}^{n} 4\pi(-1)^{m} i_n(\lambda_{\ell}r_t)Y_n^m(\theta_t,\varphi_t)\widehat{P}_n^m\Big(\frac{\lambda_{\ell z}}{\lambda_{\ell}}\Big)e^{-\ri m\alpha}
\end{split}
\end{equation}
for $\mathfrak b=1, 2$. Then, a substitution of \eqref{integrandme} into \eqref{generalcomponentsimag} gives ME
\begin{equation}\label{melayerupgoingimage}
\begin{split}
\tilde u_{\ell\ell'}^{\mathfrak{ab}}(\bs r, \bs r'_{\mathfrak{ab}})=\sum\limits_{n=0}^{\infty}\sum\limits_{m=-n}^{n}  M_{nm}^{\mathfrak{ab}}\widetilde{\mathcal F}_{nm}^{\mathfrak{ab}}(\bs r, \bs r_c^{\mathfrak{ab}}), \quad M_{nm}^{\mathfrak{ab}}=4\pi i_n(\lambda_{\ell'}r_s^{\mathfrak{ab}})\overline{Y_n^{m}(\theta_s^{\mathfrak{ab}},\varphi_s^{\mathfrak{ab}})},
\end{split}
\end{equation}
at equivalent polarization source centers $\bs r_c^{\mathfrak{ab}}$ and LE
\begin{equation}\label{lelayerimage}
\begin{split}
\tilde u_{\ell\ell'}^{\mathfrak{ab}}(\bs r, \bs r'_{\mathfrak{ab}})=\sum\limits_{n=0}^{\infty}\sum\limits_{m=-n}^{n} L_{nm}^{\mathfrak{ab}}i_n(\lambda_{\ell}r_t)Y_n^m(\theta_t,\varphi_t)
\end{split}
\end{equation}
at target center $\bs r_c^t$. Here, $\widetilde{\mathcal F}_{nm}^{\mathfrak{ab}}(\bs r, \bs r_c^{\mathfrak{ab}})$ are represented by Sommerfeld-type integrals
\begin{equation}\label{mebasis}
\begin{split}
\widetilde{\mathcal F}_{nm}^{1\mathfrak b}(\bs r, \bs r_c^{1\mathfrak b})=&\frac{1}{8\pi^2}\int_0^{\infty}\int_0^{2\pi}\frac{\lambda_{\rho}}{\lambda_{\ell z}}{\mathcal E}^{+}_{\ell\ell'}(\bs r, \bs r_c^{1\mathfrak b})\sigma_{\ell\ell'}^{1\mathfrak b}(\lambda_{\rho})\widehat{P}_{n}^m\Big(\frac{\lambda_{\ell'z}}{\lambda_{\ell'}}\Big)e^{\ri m\alpha}d\alpha d\lambda_{\rho},\\
\widetilde{\mathcal F}_{nm}^{2\mathfrak b}(\bs r, \bs r_c^{2\mathfrak b})=&\frac{(-1)^{n+m}}{8\pi^2}\int_0^{\infty}\int_0^{2\pi}\frac{\lambda_{\rho}}{\lambda_{\ell z}}{\mathcal E}^{-}_{\ell\ell'}(\bs r, \bs r_c^{2\mathfrak b})\sigma_{\ell\ell'}^{2\mathfrak b}(\lambda_{\rho})\widehat{P}_{n}^m\Big(\frac{\lambda_{\ell'z}}{\lambda_{\ell'}}\Big)e^{\ri m\alpha}d\alpha d\lambda_{\rho},
\end{split}
\end{equation}
and the local expansion coefficients are given by
\begin{equation}\label{lecoeffimage}
\begin{split}
L_{nm}^{1\mathfrak b}=&\frac{(-1)^n}{2\pi }\int_0^{\infty}\int_0^{2\pi}\frac{\lambda_{\rho}}{\lambda_{\ell z}}{\mathcal E}^{+}_{\ell\ell'}(\bs r_c^t, \bs r'_{1\mathfrak b})\sigma_{\ell\ell^{\prime}}^{1\mathfrak b}(\lambda_{\rho})\widehat{P}_{n}^m\Big(\frac{\lambda_{\ell z}}{\lambda_{\ell}}\Big)e^{-\ri m\alpha}d\alpha d\lambda_{\rho},\\
L_{nm}^{2\mathfrak b}=&\frac{(-1)^{m}}{2\pi}\int_0^{\infty}\int_0^{2\pi}\frac{\lambda_{\rho}}{\lambda_{\ell z}}{\mathcal E}^{-}_{\ell\ell'}(\bs r_c^t, \bs r'_{2\mathfrak b})\sigma_{\ell\ell^{\prime}}^{2\mathfrak b}(\lambda_{\rho})\widehat{P}_{n}^m\Big(\frac{\lambda_{\ell z}}{\lambda_{\ell}}\Big)e^{-\ri m\alpha}d\alpha d\lambda_{\rho}.
\end{split}
\end{equation}
According to the definition of $\mathcal E_{\ell\ell'}^-(\bs r, \bs r')$ and $\mathcal E_{\ell\ell'}^+(\bs r, \bs r')$  in \eqref{expkernelexp},  the centers $\bs r_c^t$ and $\bs r_c^{\mathfrak{ab}}$ have to satisfy
\begin{equation}\label{imagecentercond}
z_c^{1\mathfrak b}<d_{\ell}, \quad z_c^{2\mathfrak b}>d_{\ell-1},\quad z_c^t>d_{\ell} \;\; {\rm for}\;\; \tilde u_{\ell\ell'}^{1\mathfrak b}(\bs r, \bs r'_{1\mathfrak b}); \quad z_c^t<d_{\ell-1} \;\; {\rm for}\;\; \tilde u_{\ell\ell'}^{2\mathfrak b}(\bs r, \bs r'_{2\mathfrak b}).
\end{equation}
to ensure the exponential decay in $\mathcal E_{\ell\ell'}^+(\bs r, \bs r^{1\mathfrak b}_c), \mathcal E_{\ell\ell'}^-(\bs r, \bs r^{2\mathfrak b}_c)$ and $\mathcal E_{\ell\ell'}^{+}(\bs r_c^t, \bs r'_{1\mathfrak b}), \mathcal E_{\ell\ell'}^{-}(\bs r_c^t, \bs r'_{2\mathfrak b})$ as $\lambda_{\rho}\rightarrow\infty$ and hence the convergence of the corresponding Sommerfeld-type integrals in \eqref{mebasis} and \eqref{lecoeffimage}. These restriction can be met in practice, since we are considering targets in the $\ell$-th layer and the equivalent polarization coordinates are always locate up the interface $z=d_{\ell-1}$ or below the interface $z=d_{\ell}$.


Next, we discuss the center shifting and translation for ME \eqref{melayerupgoingimage} and LE \eqref{lelayerimage}. A desirable feature of the expansions of reaction components discussed above is that the formula \eqref{melayerupgoingimage} for the ME coefficients and the formula \eqref{lelayerimage} for the LE have exactly the same form as the formulas of ME coefficients and LE for free space Green's function. Therefore, we can see that center shifting for MEs and LEs are exactly the same as free space case given in \eqref{metome}.

We only need to derive the translation operator from ME \eqref{melayerupgoingimage} to LE \eqref{lelayerimage}. Recall the definition of exponential functions in \eqref{mekernelimage}, ${\mathcal E}^{+}_{\ell\ell'}(\bs r, \bs r_c^{1\mathfrak b})$ and ${\mathcal E}^{-}_{\ell\ell'}(\bs r, \bs r_c^{2\mathfrak b})$ have the following splitting
\begin{equation*}
\begin{split}
{\mathcal E}^{+}_{\ell\ell'}(\bs r, \bs r_c^{1\mathfrak{b}})&={\mathcal E}^{+}_{\ell\ell'}(\bs r_c^t, \bs r_c^{1\mathfrak b})e^{\ri\bs\lambda_{\alpha}\cdot(\bs\rho-\bs\rho_c^t)}e^{-\lambda_{\ell z} (z-z_c^t)},\\
{\mathcal E}^{-}_{\ell\ell'}(\bs r, \bs r_c^{2\mathfrak b})&={\mathcal E}^{-}_{\ell\ell'}(\bs r_c^t, \bs r_c^{2\mathfrak b})e^{\ri\bs\lambda_{\alpha}\cdot(\bs\rho-\bs\rho_c^t)}e^{\lambda_{\ell z} (z-z_c^t)}.
\end{split}
\end{equation*}
Applying spherical harmonic expansion \eqref{extmodifiedfunkhecke} again, we obtain
\begin{equation*}
\begin{split}
e^{\ri\bs\lambda_{\alpha}\cdot(\bs\rho-\bs\rho_c^t)+ \lambda_{\ell z} (z-z_c^t)}=4\pi\sum\limits_{n=0}^{\infty}\sum\limits_{m=-n}^{n} (- 1)^{m} i_n(r_t)Y_n^m(\theta_t,\varphi_t)\widehat{P}_n^m\Big(\frac{\lambda_{\ell z}}{\lambda_{\ell}}\Big)e^{-\ri m\alpha},\\
e^{\ri\bs\lambda_{\alpha}\cdot(\bs\rho-\bs\rho_c^t)-\lambda_{\ell z} (z-z_c^t)}=4\pi\sum\limits_{n=0}^{\infty}\sum\limits_{m=-n}^{n} (- 1)^{n} i_n(r_t)Y_n^m(\theta_t,\varphi_t)\widehat{P}_n^m\Big(\frac{\lambda_{\ell z}}{\lambda_{\ell}}\Big)e^{-\ri m\alpha}.
\end{split}
\end{equation*}
Substituting into \eqref{melayerupgoingimage}, the multipole expansion is translated to local expansion \eqref{lelayerimage} via
\begin{equation}\label{metoleimage1}
L_{nm}^{1\mathfrak{b}}=\sum\limits_{\nu=0}^{\infty}\sum\limits_{|\mu|=0}^{\nu}T_{nm,\nu\mu}^{1\mathfrak{b}}M_{\nu\mu}^{1\mathfrak{b}},\quad L_{nm}^{2\mathfrak{b}}=\sum\limits_{\nu=0}^{\infty}\sum\limits_{|\mu|=0}^{\nu}T_{nm,\nu\mu}^{2\mathfrak{b}}M_{\nu\mu}^{2\mathfrak{b}},
\end{equation}
and the multipole-to-local translation operators are given in integral forms as follows
\begin{equation}\label{metoleimage2}
\begin{split}
T_{nm,\nu\mu}^{1\mathfrak{b}}=&\frac{(-1)^{n}}{2\pi}\int_0^{\infty}\int_0^{2\pi}\frac{\lambda_{\rho}}{\lambda_{\ell z}}{\mathcal E}^{+}(\bs r_c^t, \bs r_c^{1\mathfrak{b}})\sigma_{\ell\ell'}^{1\mathfrak{b}}(\lambda_{\rho})Q_{nm}^{\nu\mu}(\lambda_{\rho})e^{\ri (\mu-m)\alpha}d\alpha d\lambda_{\rho},\\
T_{nm,\nu\mu}^{2\mathfrak{b}}=&\frac{(-1)^{m+\nu+\mu}}{2\pi}\int_0^{\infty}\int_0^{2\pi}\frac{\lambda_{\rho}}{\lambda_{\ell z}}{\mathcal E}^{-}(\bs r_c^t, \bs r_c^{2\mathfrak{b}})\sigma_{\ell\ell'}^{2\mathfrak{b}}(\lambda_{\rho})Q_{nm}^{\nu\mu}(\lambda_{\rho})e^{\ri (\mu-m)\alpha}d\alpha d\lambda_{\rho},
\end{split}
\end{equation}
where
$$Q_{nm}^{\nu\mu}(\lambda_{\rho})=\widehat{P}_n^m\Big(\frac{\lambda_{\ell z}}{\lambda_{\ell}}\Big)\widehat{P}_{\nu}^{\mu}\Big(\frac{\lambda_{\ell' z}}{\lambda_{\ell'}}\Big).$$
Again the convergence of the Sommerfeld-type integrals in \eqref{metoleimage2} is ensured by the conditions in \eqref{imagecentercond}.

At the end of this subsection, we give some numerical examples to show the convergence behavior of the multipole expansions in \eqref{melayerupgoingimage}. Consider the multipole expansions of $\tilde u_{11}^{11}(\bs r, \bs r'_{11})$ and $\tilde u_{11}^{22}(\bs r, \bs r'_{22})$ in a three-layer media with $\varepsilon_0=1.0$, $\varepsilon_1=8.6$, $\varepsilon_2=20.5$, $\lambda_0=1.2$, $\lambda_1=0.5$, $\lambda_2=2.1$, $d_0=0$, $d_1=-1.2$. In all the following examples, we fix $\bs r'=(0.625, 0.5, -0.1)$ in the middle layer and use definition \eqref{eqpolarizedsource} to determine $\bs r_{11}^{\prime}=(0.625, 0.5, -2.3)$, $\bs r_{22}^{\prime}=(0.625, 0.5, 0.1)$. The centers for multipole expansions are set to be  $\bs r_c^{11}=(0.6, 0.6, -2.4)$, $\bs r_c^{22}=(0.6, 0.6, 0.2)$ which implies $|\bs r_{11}'-\bs r_c^{11}|=|\bs r_{22}'-\bs r_c^{22}|\approx0.1436$. For both components, we shall test three targets given as follows
\begin{equation*}
\bs r_1=(0.5, 0.625, -0.1),\quad \bs r_2= (0.5, 0.625, -0.6), \quad \bs r_3=(0.5, 0.625, -1.1);
\end{equation*}
The relative errors against truncation number $p$ are depicted in Fig. \ref{meconvergence}. We also plot the convergence rates similar with that of the multipole expansion of free space Green's function, i.e., $O\Big[\Big(\frac{|\bs r-\bs r_c^{\mathfrak{ab}}|}{|\bs r_{\mathfrak{ab}}'-\bs r_c^{\mathfrak{ab}}|}\Big)^{p+1}\Big]$ as reference convergence rates. The results clearly show that ME of reaction components $u_{11}^{11}(\bs r, \bs r'_{11})$ and $u_{11}^{22}(\bs r, \bs r'_{22})$have spectral convergence rate $O\Big[\Big(\frac{|\bs r-\bs r_c^{\mathfrak{ab}}|}{|\bs r_{\mathfrak{ab}}'-\bs r_c^{\mathfrak{ab}}|}\Big)^{p+1}\Big]$ similar as that of free space Green's function. Therefore, the ME \eqref{melayerupgoingimage} can be used to develop FMM for efficient computation of the reaction components as in the development of Yukawa-FMM for free space Green's function.
\begin{figure}[ptbh]
	\center
	\subfigure[$\tilde u_{11}^{11}(\bs r, \bs r'_{11})$]{\includegraphics[scale=0.40]{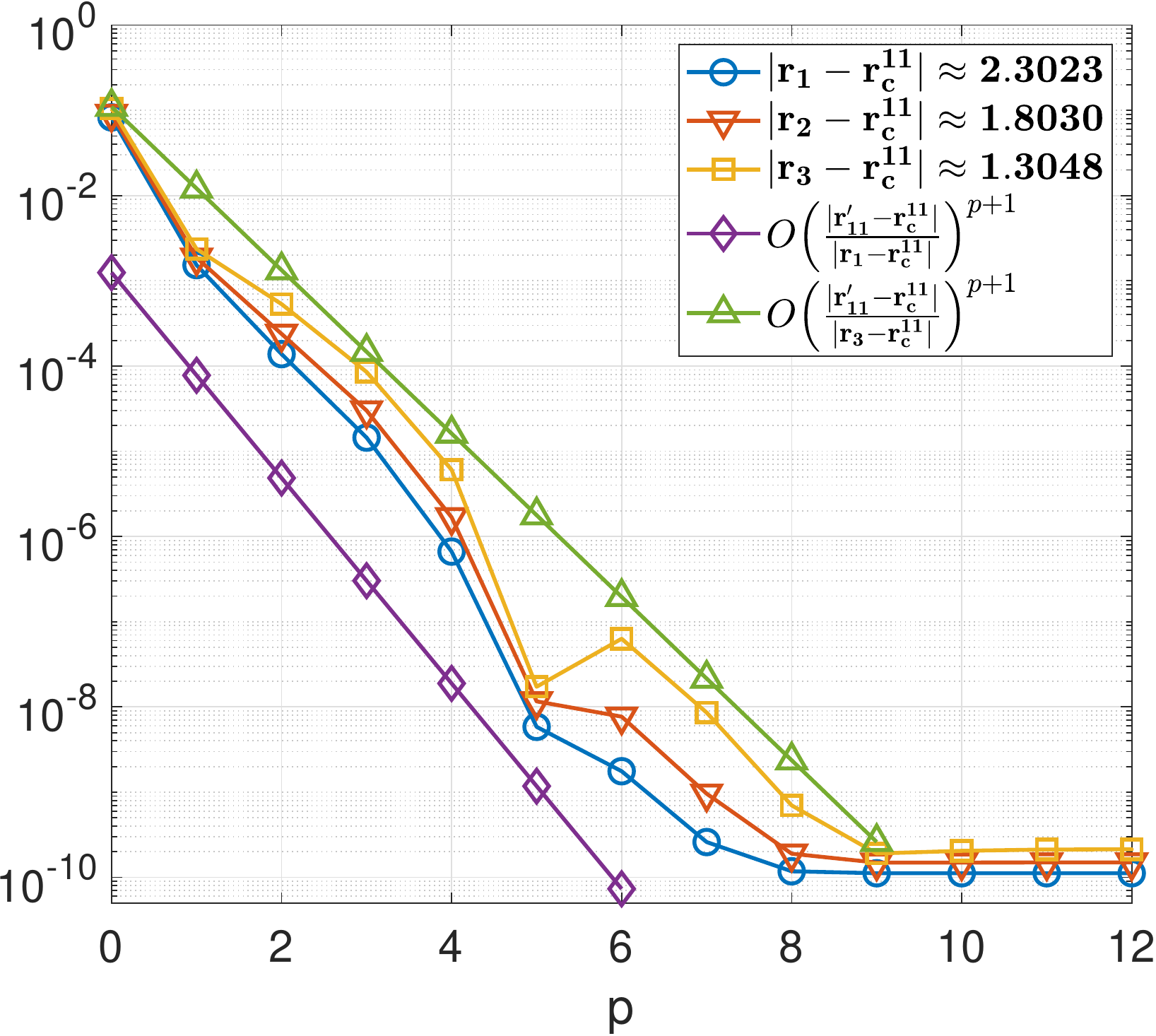}}\qquad
	\subfigure[$\tilde u_{11}^{22}(\bs r, \bs r'_{22})$]{\includegraphics[scale=0.40]{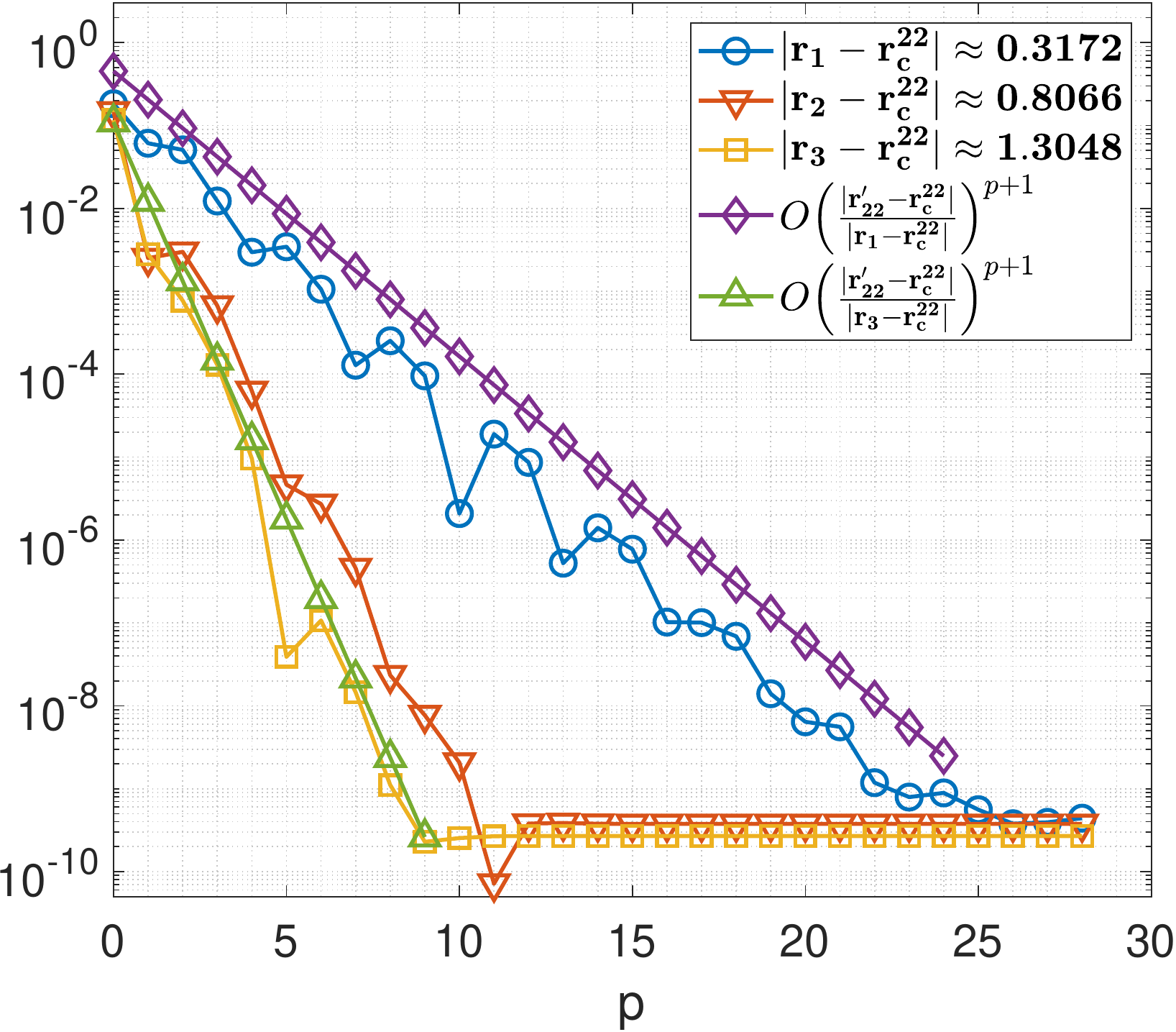}}
	\caption{Spectral convergence of the multipole expansions for reaction components.}%
	\label{meconvergence}%
\end{figure}


\subsection{FMM algorithm and efficient calculation of Sommerfeld-type integrals}

The framework of the traditional FMM together with ME \eqref{melayerupgoingimage}, LE \eqref{lelayerimage}, M2L translation \eqref{metoleimage1}-\eqref{metoleimage2} and free space ME and LE center shifting \eqref{metome} constitute the FMM for the computation of reaction components $\Phi_{\ell\ell'}^{\mathfrak{ab}}(\bs r_{\ell i})$, $\mathfrak a, \mathfrak b=1, 2$. In the FMM for each reaction component, a large box is defined to include all equivalent polarization coordinates and corresponding target particles where the adaptive tree structure will be built by a bisection procedure, see. Fig. \ref{polarizedsource} (right). { Note that the validity of the ME \eqref{melayerupgoingimage}, LE \eqref{lelayerimage} and M2L translation \eqref{metoleimage1} used in the algorithm imposes restrictions \eqref{imagecentercond} on the centers, accordingly. This can be ensured by setting the largest box for the specific reaction component to be equally divided by the interface between equivalent polarization coordinates and targets, see. Fig. \ref{polarizedsource}. Thus, the largest box for the FMM implementation will be different for different reaction components.  With this setting, all source and target boxes of level higher than zeroth level in the adaptive tree structure will have centers below or above the interfaces, accordingly. The fast multipole algorithm for the computation of the reaction component $\Phi_{\ell\ell'}^{\mathfrak{ab}}(\bs r_{\ell i})$ is summarized in Algorithm 1.
All the interactions given by \eqref{totalinteraction} will be obtained by first calculating all components and then summing them up. The algorithm is presented in Algorithm 2.
	
The double integrals involved in the ME, LE and M2L translations can be simplified by using the following identity
\begin{equation}
J_n(z)=\frac{1}{2\pi \ri^n}\int_0^{2\pi}e^{\ri z\cos\theta+\ri n\theta}d\theta.
\end{equation}
Define
\begin{equation}\label{zexpkernel}
\mathcal Z_{\ell\ell'}^+(z, z'):=e^{-\lambda_{\ell z}(z-d_{\ell})-\lambda_{\ell'z}(d_{\ell}-z^{\prime})},\quad \mathcal Z_{\ell\ell'}^-(z, z'):=e^{-\lambda_{\ell z}(d_{\ell-1}-z)-\lambda_{\ell'z}(z^{\prime}-d_{\ell-1})}.
\end{equation}
Then, the multipole expansion functions in \eqref{mebasis} can be simplified as
\begin{equation*}
\begin{split}
\widetilde{\mathcal F}_{nm}^{1\mathfrak b}(\bs r, \bs r_c^{1\mathfrak b})=&\frac{e^{\ri m\varphi_s^{1\mathfrak b}}}{4\pi}\int_0^{\infty}\lambda_{\rho}J_m(\lambda_{\rho}\rho_s^{1\mathfrak b})\frac{\mathcal Z_{\ell\ell'}^+(z, z_c^{1\mathfrak b})}{\lambda_{\ell z}}\sigma_{\ell\ell'}^{1\mathfrak b}(\lambda_{\rho})\ri^m\widehat{P}_{n}^m\Big(\frac{\lambda_{\ell'z}}{\lambda_{\ell'}}\Big)d\lambda_{\rho},\\
\widetilde{\mathcal F}_{nm}^{2\mathfrak b}(\bs r, \bs r_c^{2\mathfrak b})=&\frac{(-1)^{n+m}e^{\ri m\varphi_s^{2\mathfrak b}}}{4\pi}\int_0^{\infty}\lambda_{\rho}J_m(\lambda_{\rho}\rho_s^{2\mathfrak b})\frac{\mathcal Z_{\ell\ell'}^-(z, z_c^{2\mathfrak b})}{\lambda_{\ell z}}\sigma_{\ell\ell'}^{2\mathfrak b}(\lambda_{\rho})\ri^m\widehat{P}_{n}^m\Big(\frac{\lambda_{\ell'z}}{\lambda_{\ell'}}\Big)d\lambda_{\rho},
\end{split}
\end{equation*}
\begin{algorithm}\label{algorithm1}
	\caption{FMM for general reaction component $\Phi_{\ell\ell'}^{\mathfrak{ab}}(\bs r_{\ell i}), i=1, 2, \cdots, N_{\ell}$}
	\begin{algorithmic}
		\State Determine $z$-coordinates of equivalent polarization sources for all source particles.
		\State Generate an adaptive hierarchical tree structure with polarization sources $\{Q_{\ell'j}, \bs r_{\ell'j}^{11}\}_{j=1}^{N_{\ell'}}$ and targets $\{\bs r_{\ell i}\}_{i=1}^{N_{\ell}}$.
		\State{\bf Upward pass:}
		\For{$l=H \to 0$}
		\For{all boxes $j$ on source tree level $l$ }
		\If{$j$ is a leaf node}
		\State{form the free-space ME using Eq. \eqref{melayerupgoingimage}.}
		\Else
		\State form the free-space ME by merging children's expansions using the free-space center shift translation operator \eqref{metome}.
		\EndIf
		\EndFor
		\EndFor
		\State{\bf Downward pass:}
		\For{$l=1 \to H$}
		\For{all boxes $j$ on target tree level $l$ }
		\State shift the LE of $j$'s parent to $j$ itself using the free-space shifting \eqref{metome}.
		\State collect interaction list contribution using the source box to target box translation operator in Eq. \eqref{metoleimage1} while $T_{nm,\nu\mu}^{\mathfrak{ab}}$ are computed using \eqref{coefintable}, \eqref{metoletableform} and forward recursion \eqref{recurrence1} for $\mathcal S_{nm,ij}^{\mathfrak{ab}}$.
		\EndFor
		\EndFor
		\State {\bf Evaluate Local Expansions:}
		\For{each leaf node (childless box)}
		\State evaluate the local expansion at each particle location.
		\EndFor
		\State {\bf Local Direct Interactions:}
		\For{$i=1 \to N$ }
		\State compute Eq. \eqref{reactcompusingpolar} of target particle $i$ in the neighboring boxes using pre-computed tables of $I_{00}^{\mathfrak{ab}}(\rho, z, z')$.
		\EndFor
	\end{algorithmic}
\end{algorithm}
and the expression \eqref{lecoeffimage} for local expansion coefficients can be simplified as
\begin{equation*}
\begin{split}
L_{nm}^{1\mathfrak b}=&(-1)^ne^{-\ri m\varphi_t^{1\mathfrak{b}}}\int_0^{\infty}\lambda_{\rho}J_{-m}(\lambda_{\rho}\rho_t^{1\mathfrak{b}})\frac{\mathcal Z_{\ell\ell'}^+(z_c^t, z_{1\mathfrak b}')}{\lambda_{\ell z}}\sigma_{\ell\ell^{\prime}}^{1\mathfrak b}(\lambda_{\rho})\ri^{-m}\widehat{P}_{n}^m\Big(\frac{\lambda_{\ell z}}{\lambda_{\ell}}\Big) d\lambda_{\rho},\\
L_{nm}^{2\mathfrak b}=&(-1)^{m}e^{-\ri m\varphi_t^{2\mathfrak{b}}}\int_0^{\infty}\lambda_{\rho}J_{-m}(\lambda_{\rho}\rho_t^{2\mathfrak{b}})\frac{\mathcal Z_{\ell\ell'}^-(z_c^t, z_{2\mathfrak b}')}{\lambda_{\ell z}}\sigma_{\ell\ell^{\prime}}^{2\mathfrak b}(\lambda_{\rho})\ri^{-m}\widehat{P}_{n}^m\Big(\frac{\lambda_{\ell z}}{\lambda_{\ell}}\Big) d\lambda_{\rho},
\end{split}
\end{equation*}
for $\mathfrak b=1, 2$, where $(\rho_s^{\mathfrak{ab}}, \varphi_s^{\mathfrak{ab}})$ and $(\rho_t^{\mathfrak{ab}}, \varphi_t^{\mathfrak{ab}})$ are polar coordinates of $\bs r-\bs r_c^{\mathfrak{ab}}$ and $\bs r_c^t-\bs r'_{\mathfrak{ab}}$ projected in $xy$ plane. Moreover, the multipole to local translation \eqref{metoleimage2} can be simplified as
\begin{equation}\label{me2lesimplified}
\begin{split}
T_{nm,\nu\mu}^{1\mathfrak{b}}=&(-1)^{n}D_{m\mu}^{(1)}(\varphi_{st}^{1\mathfrak b})\int_0^{\infty}\lambda_{\rho}J_{\mu-m}(\lambda_{\rho}\rho_{st}^{1\mathfrak b})\frac{\mathcal Z_{\ell\ell'}^+(z_c^t,z_c^{1\mathfrak b})}{\lambda_{\ell z}}Q_{nm}^{\nu\mu}(\lambda_{\rho})\sigma_{\ell\ell'}^{1\mathfrak{b}}(\lambda_{\rho})d\lambda_{\rho},\\
T_{nm,\nu\mu}^{2\mathfrak{b}}=&(-1)^{\nu}D_{m\mu}^{(2)}(\varphi_{st}^{2\mathfrak b})\int_0^{\infty}\lambda_{\rho}J_{\mu-m}(\lambda_{\rho}\rho_{st}^{2\mathfrak b})\frac{\mathcal Z_{\ell\ell'}^-(z_c^t,z_c^{2\mathfrak b})}{\lambda_{\ell z}}Q_{nm}^{\nu\mu}(\lambda_{\rho})\sigma_{\ell\ell'}^{2\mathfrak{b}}(\lambda_{\rho})d\lambda_{\rho},
\end{split}
\end{equation}
where $(\rho_{st}^{\mathfrak{ab}}, {\phi}_{st}^{\mathfrak{ab}})$ is the polar coordinates of $\bs r_c-\bs r_c^{\mathfrak{ab}}$ projected in $xy$ plane,
$$D_{m\mu}^{(1)}(\varphi)=\ri^{\mu-m}e^{\ri(\mu-m){\varphi}},\quad D_{m\mu}^{(2)}(\varphi)=(-1)^{m+\mu}\ri^{\mu-m}e^{\ri(\mu-m){\varphi}}.$$
\begin{algorithm}\label{algorithm2}
	\caption{3-D FMM for \eqref{totalinteraction}}
	\begin{algorithmic}
		\For{$\ell=0 \to L$}
		\State{use free space FMM to compute $\Phi_{\ell}^{free}(\bs r_{\ell i})$, $i=1, 2, \cdots, N_{\ell}$.}
		\EndFor
		\For{$\ell=0 \to L-1$}
		\For{$\ell'=0 \to L-1$ }
		\State use {\bf Algorithm 1} to compute $\Phi_{\ell\ell'}^{11}(\bs r_{\ell i})$, $i=1, 2, \cdots, N_{\ell}$.
		\EndFor
		\For{$\ell'=1 \to L$ }
		\State use {\bf Algorithm 1} to compute $\Phi_{\ell\ell'}^{12}(\bs r_{\ell i})$, $i=1, 2, \cdots, N_{\ell}$.
		\EndFor
		\EndFor
		\For{$\ell=1 \to L$}
		\For{$\ell'=0 \to L-1$ }
		\State use {\bf Algorithm 1} to compute $\Phi_{\ell\ell'}^{21}(\bs r_{\ell i})$, $i=1, 2, \cdots, N_{\ell}$.
		\EndFor
		\For{$\ell'=1 \to L$ }
		\State use {\bf Algorithm 1} to compute $\Phi_{\ell\ell'}^{22}(\bs r_{\ell i})$, $i=1, 2, \cdots, N_{\ell}$.
		\EndFor
		\EndFor
	\end{algorithmic}
\end{algorithm}
Define integral
\begin{equation}\label{uniformintegral}
\begin{split}
I_{nm,\nu\mu}^{1\mathfrak{b}}(\rho, z, z')=\int_0^{\infty}\lambda_{\rho}J_{\mu-m}(\lambda_{\rho}\rho)\frac{{\mathcal{Z}}_{\ell\ell'}^{+}(z, z')\sigma_{\ell\ell'}^{1\mathfrak b}(\lambda_{\rho})}{\lambda_{\ell z}}\ri^{\mu-m}Q_{nm}^{\nu\mu}(\lambda_{\rho}) d\lambda_{\rho},\\
I_{nm,\nu\mu}^{2\mathfrak{b}}(\rho, z, z')=\int_0^{\infty}\lambda_{\rho}J_{\mu-m}(\lambda_{\rho}\rho)\frac{{\mathcal{Z}}_{\ell\ell'}^{-}(z, z')\sigma_{\ell\ell'}^{2\mathfrak b}(\lambda_{\rho})}{\lambda_{\ell z}}\ri^{\mu-m}Q_{nm}^{\nu\mu}(\lambda_{\rho}) d\lambda_{\rho},
\end{split}
\end{equation}
Then
\begin{equation}\label{coefintable}
\begin{split}
&\widetilde{\mathcal F}_{nm}^{1\mathfrak b}(\bs r, \bs r_c^{1\mathfrak b})=\frac{e^{\ri m\varphi_s^{1\mathfrak b}}}{\sqrt{4\pi}}I_{00,nm}^{1\mathfrak{b}}(\rho_s^{1\mathfrak b}, z, z_c^{1\mathfrak b}),\\
& \widetilde{\mathcal F}_{nm}^{2\mathfrak b}(\bs r, \bs r_c^{2\mathfrak b})=\frac{(-1)^{n+m}e^{\ri m\varphi_s^{2\mathfrak b}}}{\sqrt{4\pi}}I_{00,nm}^{2\mathfrak{b}}(\rho_s^{2\mathfrak b}, z, z_c^{2\mathfrak b}),\\
&L_{nm}^{1\mathfrak b}=(-1)^n\sqrt{4\pi}e^{-\ri m\varphi_t^{1\mathfrak{b}}}I_{nm,00}^{1\mathfrak{b}}(\rho_t^{1\mathfrak b}, z_c^t, z'_{1\mathfrak b}),\\
& L_{nm}^{2\mathfrak b}=(-1)^{m}\sqrt{4\pi}e^{-\ri m\varphi_t^{2\mathfrak{b}}}I_{nm,00}^{2\mathfrak{b}}(\rho_t^{2\mathfrak b}, z'_{2\mathfrak b}, z_c^t),\\
& T_{nm,\nu\mu}^{1\mathfrak b}=(-1)^{n}e^{\ri(\mu-m)\varphi_{st}^{1\mathfrak b}}I_{nm,\nu\mu}^{1\mathfrak{b}}(\rho_{st}^{1\mathfrak b}, z_c^t, z_c^{1\mathfrak b}),\\
& T_{nm,\nu\mu}^{2\mathfrak b}=(-1)^{\nu+m+\mu}e^{\ri(\mu-m)\varphi_{st}^{2\mathfrak b}}I_{nm,\nu\mu}^{2\mathfrak{b}}(\rho_{st}^{2\mathfrak b}, z_c^t, z_c^{2\mathfrak b}).
\end{split}
\end{equation}

The FMM demands efficient computation of Sommerfeld-type integrals $I_{nm,\nu\mu}^{\mathfrak{ab}}$ defined in \eqref{uniformintegral}. It
is clear that they have oscillatory integrands. These integrals are convergent when the target and
source particles are not exactly on the interfaces of a layered medium.
High order quadrature rules could be used for direct
numerical computation at runtime. However, this becomes prohibitively expensive due to a
large number of integrals needed in the FMM. In fact, $O(p^{4})$ integrals
will be required for each source box to target box translation. Moreover, the
involved integrand decays more slowly as the order of the involved associated Legendre function increases.

The Sommerfeld-type integrals $I_{nm,\nu\mu}^{\mathfrak{ab}}$ involves $Q_{nm}^{\nu\mu}(\lambda_{\rho})$ the product of two associated Legendre functions. This term can be simplified by representing its polynomial part into Legendre polynomials. Define
\begin{equation}
c_{nm}=\sqrt{\frac{2n+1}{4\pi}\frac{(n-m)!}{(n+m)!}},\quad  a_{nm}^j=\frac{(-1)^{n-j}(2j)!c_{nm}}{2^nj!(n-j)!(2j-n-m)!}.
\end{equation}
and
\begin{equation}
b_{nm}^s=\sum\limits_{j=q}^{n}\frac{(-1)^{s}a_{nm}^j(j-r)!}{s!(j-r-s)!},\quad q=\max\Big(\Big\lceil\frac{n+m}{2}\Big\rceil, s+r\Big).
\end{equation}
The derivation in \cite{wang2019fast} gives
\begin{equation}
\widehat P_n^m\Big(\frac{\lambda_{\ell z}}{\lambda_{\ell}}\Big)\widehat P_{\nu}^{\mu}\Big(\frac{\lambda_{\ell^{\prime} z}}{\lambda_{\ell^{\prime}}}\Big)=\sum\limits_{s=0}^{n-r+\nu-r'}C_{n\nu m\mu}^s\lambda_{\rho}^{|m|+|\mu|+2s}\Big(\frac{\lambda_{\ell}}{\lambda_{\ell z}}\Big)^{i}\Big(\frac{\lambda_{\ell'}}{\lambda_{\ell' z}}\Big)^{j},
\end{equation}
for all $n, \nu=0, 1, \cdots, $ and $-n\leq m\leq n$, $-\nu\leq \mu\leq \nu$ where $i=(n+|m|)(\bmod\; 2)$, $j=(\nu+|\mu|)(\bmod\; 2)$, $r=\big\lfloor (n+|m|)/2\big\rfloor, \quad r'=\big\lfloor(\nu+|\mu|)/2\big\rfloor,$
and
$$
C_{n\nu m\mu}^{s}=\sum\limits_{t=\max(s-\nu+r', 0)}^{\min(s, n-r)}\frac{\tau_m\tau_{\mu}b_{n|m|}^tb_{\nu|\mu|}^{s-t}(|m|+|\mu|+2s)!}{(\ri\lambda_{\ell})^{|m|+2t}(\ri \lambda_{\ell'})^{|\mu|+2(s-t)}}, \quad \tau_{\nu}=\begin{cases}
1, & \nu\geq 0,\\
(-1)^{-\nu}, & \nu<0.
\end{cases}
$$
Define integrals
\begin{equation}\label{M2Ltable}
\begin{split}
\mathcal S_{nm,ij}^{1\mathfrak b}(\rho, z, z')=\int_0^{\infty}\frac{\lambda_{\rho}^nJ_{m}(\lambda_{\rho}\rho)\mathcal{Z}_{\ell\ell'}^{+}(z, z')}{\sqrt{(n+m)!(n-m)!}}\frac{\sigma_{\ell\ell'}^{1\mathfrak b}(\lambda_{\rho})}{\lambda_{\ell z}}\Big(\frac{\lambda_{\ell}}{\lambda_{\ell z}}\Big)^i\Big(\frac{\lambda_{\ell'}}{\lambda_{\ell' z}}\Big)^j d\lambda_{\rho},\\
\mathcal S_{nm,ij}^{2\mathfrak b}(\rho, z, z')=\int_0^{\infty}\frac{\lambda_{\rho}^nJ_{m}(\lambda_{\rho}\rho)\mathcal{Z}_{\ell\ell'}^{-}(z, z')}{\sqrt{(n+m)!(n-m)!}}\frac{\sigma_{\ell\ell'}^{2\mathfrak b}(\lambda_{\rho})}{\lambda_{\ell z}}\Big(\frac{\lambda_{\ell}}{\lambda_{\ell z}}\Big)^i\Big(\frac{\lambda_{\ell'}}{\lambda_{\ell' z}}\Big)^j d\lambda_{\rho},
\end{split}
\end{equation}
for $i, j=0, 1$. Then
\begin{equation}\label{metoletableform}
I_{nm,\nu\mu}^{\mathfrak{ab}}(\rho, z, z')=\ri^{\mu-m}\sum\limits_{s=0}^{n-r+\nu-r'}\widetilde C_{n\nu m\mu}^s\mathcal S_{|m|+|\mu|+2s+1,\mu-m,ij}^{\mathfrak{ab}}(\rho, z, z'),
\end{equation}
where
$$\widetilde C_{n\nu m\mu}^s=\sqrt{(|m|-m+|\mu|+\mu+2s+1)£¡(|m|+m+|\mu|-\mu+2s+1)£¡}C_{n\nu m\mu}^s,$$
$i=(n+|m|)(\bmod\; 2)$, $j=(\nu+|\mu|)(\bmod\; 2)$.

An important aspect of the implementation of FMM concerns scaling. Since $M_{nm}^{\mathfrak{ab}}\approx(|\bs r-\bs r_c^{\mathfrak{ab}}|)^n$, $L_{nm}^{\mathfrak{ab}}\approx(|\bs r^{\mathfrak{ab}}-\bs r_c^t|)^{-n}$, a naive use of the expansions \eqref{melayerupgoingimage} and \eqref{lelayerimage} in the implementation of FMM is likely to encounter underflow and overflow issues. To avoid this, one must scale expansions, replacing $M_{nm}^{\mathfrak{ab}}$ by $M_{nm}^{\mathfrak{ab}}/S^n$ and $L_{nm}^{\mathfrak{ab}}$ by $L_{nm}^{\mathfrak{ab}}\cdot S^n$ where $S$ is the scaling factor. To compensate for this scaling, we replace $\widetilde{\mathcal F}_{nm}^{\mathfrak{ab}}(\bs r, \bs r_c^{\mathfrak{ab}})$ with $\widetilde{\mathcal F}_{nm}^{\mathfrak{ab}}(\bs r, \bs r_c^{\mathfrak{ab}})\cdot S^n$, $T_{nm,n'm'}^{\mathfrak{ab}}$ with $T_{nm,n'm'}^{\mathfrak{ab}}\cdot S^{n+n'}$.  Usually, the scaling factor $S$ is chosen to be the size of the box in which the computation occurs. Therefore, the following scaled Sommerfeld-type integrals
\begin{equation}
\begin{split}
S^{n}\mathcal S_{nm,ij}^{1\mathfrak{b}}(\rho, z, z')=S^{n}\int_0^{\infty}\frac{\lambda_{\rho}^nJ_{m}(\lambda_{\rho}\rho)\mathcal{Z}_{\ell\ell'}^{+}(z, z')}{\sqrt{(n+m)!(n-m)!}}\frac{\sigma_{\ell\ell'}^{1\mathfrak{b}}(\lambda_{\rho})}{\lambda_{\ell z}}\Big(\frac{\lambda_{\ell}}{\lambda_{\ell z}}\Big)^i\Big(\frac{\lambda_{\ell'}}{\lambda_{\ell' z}}\Big)^j d\lambda_{\rho},\\
S^{n}\mathcal S_{nm,ij}^{2\mathfrak{b}}(\rho, z, z')=S^{n}\int_0^{\infty}\frac{\lambda_{\rho}^nJ_{m}(\lambda_{\rho}\rho)\mathcal{Z}_{\ell\ell'}^{-}(z, z')}{\sqrt{(n+m)!(n-m)!}}\frac{\sigma_{\ell\ell'}^{2\mathfrak{b}}(\lambda_{\rho})}{\lambda_{\ell z}}\Big(\frac{\lambda_{\ell}}{\lambda_{\ell z}}\Big)^i\Big(\frac{\lambda_{\ell'}}{\lambda_{\ell' z}}\Big)^j d\lambda_{\rho},
\end{split}
\end{equation}
for all $n\geq m\geq  0$ are computed in the implementation. Recall the recurrence formula
$$J_{m+1}(z)=\frac{2m}{z}J_{m}(z)-J_{m-1}(z),$$
and define $a_n=\sqrt{n(n+1)}$, we have
\begin{equation*}
\begin{split}
S^{n}\mathcal S_{nm+1,ij}^{\mathfrak{ab}}(\rho, z, z')=&\int_0^{\infty}\frac{(\lambda_{\rho}S)^{n}J_{m+1}(\lambda_{\rho}\rho)\mathcal{Z}_{\ell\ell'}^{\pm}(z, z')}{\sqrt{(n+m+1)!(n-m-1)!}}\frac{\sigma_{\ell\ell'}^{\mathfrak{ab}}(\lambda_{\rho})}{\lambda_{\ell z}}\Big(\frac{\lambda_{\ell}}{\lambda_{\ell z}}\Big)^i\Big(\frac{\lambda_{\ell'}}{\lambda_{\ell' z}}\Big)^j d\lambda_{\rho}\\
=&\frac{2m S}{a_{n+m}\rho  }\int_0^{\infty}\frac{(\lambda_{\rho}S)^{n-1}J_{m}(\lambda_{\rho}\rho)\mathcal{Z}_{\ell\ell'}^{\pm}(z, z')}{\sqrt{(n+m-1)!(n-m-1)!}}\frac{\sigma_{\ell\ell'}^{\mathfrak{ab}}(\lambda_{\rho})}{\lambda_{\ell z}}\Big(\frac{\lambda_{\ell}}{\lambda_{\ell z}}\Big)^i\Big(\frac{\lambda_{\ell'}}{\lambda_{\ell' z}}\Big)^j{\rm d}\lambda_{\rho}\\
-&\frac{a_{n-m}}{a_{n+m}}\int_0^{\infty}\frac{(\lambda_{\rho}S)^{n}J_{m-1}(\lambda_{\rho}\rho)\mathcal{Z}_{\ell\ell'}^{\pm}(z, z')}{\sqrt{(n+m-1)!(n-m+1)!}}\frac{\sigma_{\ell\ell'}^{\mathfrak{ab}}(\lambda_{\rho})}{\lambda_{\ell z}}\Big(\frac{\lambda_{\ell}}{\lambda_{\ell z}}\Big)^i\Big(\frac{\lambda_{\ell'}}{\lambda_{\ell' z}}\Big)^j{\rm d}\lambda_{\rho},
\end{split}
\end{equation*}
which directly gives the following forward recurrence formula
\begin{equation}\label{recurrence1}
S^{n}\mathcal S_{nm+1,ij}^{\mathfrak{ab}}=\frac{2m}{a_{n+m}}\frac{S}{\rho}S^{n-1}\mathcal S_{n-1m,ij}^{\mathfrak{ab}}-\frac{a_{n-m}}{a_{n+m}}S^{n}\mathcal S_{nm-1,ij}^{\mathfrak{ab}},\quad n\geq m\geq 1.
\end{equation}
Obviously, this recurrence formula is stable if
\begin{equation}
\frac{2m}{a_{n+m}}<\frac{\rho}{S}.
\end{equation}
Define $S_{\rho}=\rho/S$, the above inequality can rewritten as
\begin{equation}
S_{\rho}^2n^2+(2m+1)S_{\rho}^2n+(S_{\rho}^2-4)m^2+mS_{\rho}^2>0.
\end{equation}
Solving the inequality for $n$ with fixed $m$, we obtain
\begin{equation}
n>2\sqrt{\frac{m^2}{S_{\rho}^2}+1}-m-\frac{1}{2}.
\end{equation}
Conversely, for
\begin{equation}
n\leq 2\sqrt{\frac{m^2}{S_{\rho}^2}+1}-m-\frac{1}{2},
\end{equation}
the backward recursion
\begin{equation}\label{backwardrecurrence1}
S^{n-1}\mathcal S_{n-1m,ij}^{\mathfrak{ab}}=\frac{a_{n+m}}{2m}\frac{\rho}{S}S^{n}\mathcal S_{nm+1,ij}^{\mathfrak{ab}}+\frac{a_{n-m}}{2m}\frac{\rho}{S}S^n\mathcal S_{nm-1,ij}^{\mathfrak{ab}},
\end{equation}
will be adopted.

Let us first consider the computation of the integrals involved in the M2L translation matrices $T_{nm,n'm'}^{\mathfrak{ab}}$.
For any polarization source box in the interaction list of a given target box, one can find that $\rho_{ts}^{\mathfrak{ab}}$ is either $0$ or larger than the box size $S$.  If $\rho_{ts}^{\mathfrak{ab}}=0$, we directly have
\begin{equation}
S^{n}\mathcal S_{nm,ij}^{\mathfrak{ab}}(0,z, z')=0, \quad \forall n\geq m>0,
\end{equation}
for any $z$ and $z'$ such that the integrals are convergent. In all other cases, we have $\rho_{ts}^{\mathfrak{ab}}\geq S$ and the forward recurrence formula \eqref{recurrence1} can always be applied as we have
$$\frac{2m}{\sqrt{(n+m+1)(n+m)}}<\frac{1}{\sqrt{3}}<\frac{\rho_{ts}^{\mathfrak{ab}}}{S},\quad n\geq m+1,\quad m\geq 1.$$
If the distribution of particles in the problem is not uniform, adaptive tree structure is usually used in the implementation of FMM for better performance. In these adaptive versions, computation of LE coefficients and potential directly using \eqref{lecoeffimage} and ME \eqref{melayerupgoingimage} will be performed if sophisticated conditions are satisfied (cf. \cite{debuhr2016dashmm}). In the computation of $\widetilde{\mathcal F}_{nm}^{\mathfrak{ab}}(\bs r, \bs r_c^{\mathfrak{ab}})\cdot S^n$ and $L_{nm}^{\mathfrak{ab}}\cdot S^n$, $\rho^{\mathfrak{ab}}_s$ and $\rho_t^{\mathfrak{ab}}$ could be arbitrary small. Therefore, the backward recurrence formula \eqref{backwardrecurrence1} is needed. Nevertheless, these direct computations are seldomly used even in the FMM with adaptive tree structure.

Given truncation number $p$, the initial values $ \{\mathcal S_{n0,ij}^{\mathfrak{ab}}(\rho, z, z')\}_{n=0}^{2p+3}$ and $\{ \mathcal S_{n1,ij}^{\mathfrak{ab}}(\rho, z, z')\}_{n=1}^{2p+3}$ for forward recursion or the initial values $ \{\mathcal S_{(2p+3)m,ij}^{\mathfrak{ab}}(\rho, z, z')\}_{m=0}^{2p+3}$ for backward recursion are computed by using composite Gaussian quadrature along the positive real line.
We truncate the unbounded interval $[0, \infty)$ at a point $X_{max}>0$, where the integrand has
decayed to a user specified tolerance (e.g., $1.0e-14$).

\begin{rem}\label{remark1}
	In the computation of a general reaction component $\Phi^{\mathfrak{ab}}_{\ell\ell'}(\bs r_{\ell i}), i=1, 2, \cdots, N_{\ell}$, the targets and equivalent polarization sources will locate at different sides of the material interface $z=d_{\ell-1}$ if $\mathfrak a=1$ or $z=d_{\ell}$ otherwise. Therefore, most target boxes if not all in the leaves of the target tree are far away from all source boxes in the leaves of the source tree. Usually, no direct interactions between sources and targets are calculated once the size of the smallest box is smaller than the minimum distance between sources and the corresponding interface. That means the time consuming computation of integrals $I_{00}^{\mathfrak{ab}}(\rho, z)$ for direct interaction is seldomly performed in the FMM. By the same reason, the interaction list of most target boxes in the target tree are empty. Therefore, the number of M2L translations in the FMM for reaction components are much less than that in the classic FMM for free space problems. To make it clear, we give an illustration in Fig. \ref{treestructure} using a 2-D tree structure for reaction component $\Phi_{\ell\ell'}^{22}(\bs r_{\ell i})$. The number of sources boxes in the interaction list of all target boxes in the fourth layer are counted (see the numbers in the 3rd subplot in Fig. \ref{treestructure}). Above discussions explain the fact that the FMMs for reaction components are much more efficient than FMM for free space components when the number of sources and targets is large enough.
	\begin{figure}[ht!]
		\centering
		\includegraphics[scale=0.7]{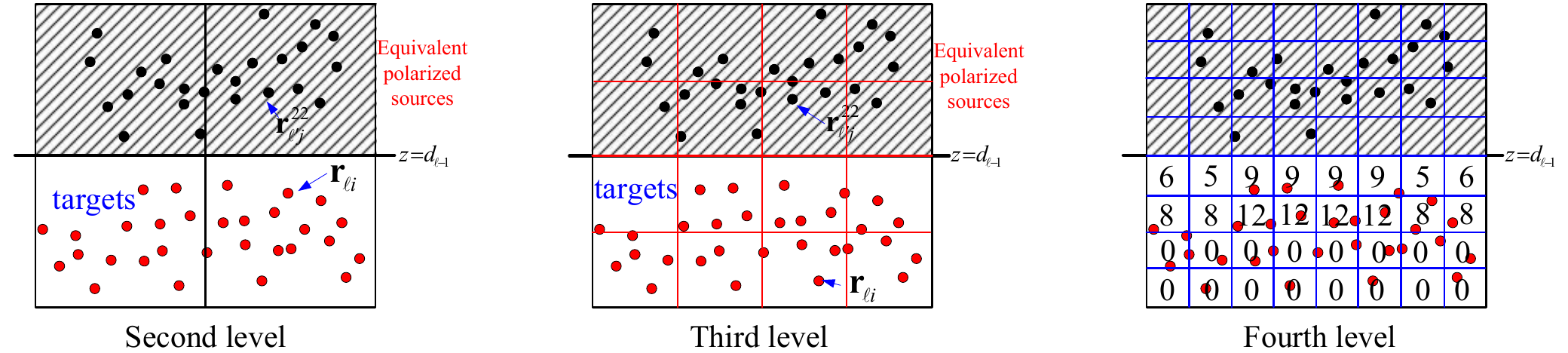}
		\vspace{-15pt}
		\caption{Boxes in the source (shadowed) and target tree for the computation of  $\Phi_{\ell\ell'}^{22}(\bs r_{\ell i})$.}
		\label{treestructure}
	\end{figure}
\end{rem}

\section{Numerical results}

In this section, we present numerical results to demonstrate the performance
of the proposed FMM for linearized Poisson-Boltzmann equation in layered media.
This algorithm is implemented based on an open-source adaptive FMM package DASHMM
\cite{debuhr2016dashmm} on a
workstation with two Xeon E5-2699 v4 2.2 GHz processors (each has 22 cores)
and 500GB RAM using the gcc compiler version 6.3.

We test the problem in three layers media with interfaces placed at $z_{0}=0$, $z_{1}=-1.2$. Particles
are set to be uniformly distributed in irregular domains which are obtained by shifting the domain determined by $r=0.5-a+\frac{a}{8}(35\cos^{4}\theta-30\cos^{2}\theta+3)$
with $a=0.1,0.15,0.05$ to new centers $(0,0,0.6)$, $(0,0,-0.6)$
and $(0,0,-1.8)$, respectively (see
Fig. \ref{fmmperformance} (a) for the cross section of the domains). All particles are generated by keeping the
uniform distributed particles in a larger cube within corresponding irregular
domains. In the layered media, the dielectric constant $\{\varepsilon_{\ell}\}_{\ell=0}^2$ and the inverse Debye-Huckel length $\{\lambda_{\ell}\}_{\ell=0}^2$ are set to be
$$\varepsilon_0=1.0, \quad\varepsilon_1=8.6, \quad \varepsilon_2=20.5,\quad \lambda_0=1.2,\quad \lambda_1=0.5, \quad \lambda_2=2.1.$$
Let $\widetilde{\Phi}_{\ell}%
(\boldsymbol{r}_{\ell i})$ be the approximated values of $\Phi_{\ell
}(\boldsymbol{r}_{\ell i})$ calculated by FMM. Define $\ell^{2}$ and maximum errors as
\begin{equation}
Err_{2}^{\ell}:=\sqrt{\frac{\sum\limits_{i=1}^{N_{\ell}}|\Phi_{\ell
		}(\boldsymbol{r}_{\ell i})-\widetilde{\Phi}_{\ell}(\boldsymbol{r}_{\ell
			i})|^{2}}{\sum\limits_{i=1}^{N_{\ell}}|\Phi_{\ell}(\boldsymbol{r}_{\ell
			i})|^{2}}},\qquad Err_{max}^{\ell}:=\max\limits_{1\leq i\leq{N_{\ell}}}%
\frac{|\Phi_{\ell}(\boldsymbol{r}_{\ell i})-\widetilde{\Phi}_{\ell
	}(\boldsymbol{r}_{\ell i})|}{|\Phi_{\ell}(\boldsymbol{r}_{\ell i})|}.
\end{equation}
For accuracy test, we put
$N=912+640+1296$ particles in the irregular domains in three layers see Fig. \ref{fmmperformance} (a). Convergence rates against $p$ are depicted in Fig. \ref{fmmperformance} (b).
The CPU time for the computation of all three free space components $\{\Phi^{free}_{\ell}(\boldsymbol{r}_{\ell i})\}_{\ell=0}^2$, three selected reaction components
$\{\Phi^{11}_{00},\Phi^{11}_{11},\Phi^{22}_{22}\}$ and all sixteen reaction components $\Phi^{\mathfrak{ab}}_{\ell\ell'}(\boldsymbol{r}_{\ell i})$ with truncation $p=5$ are compared in Fig. \ref{fmmperformance} (c) for up to 3 millions particles. It shows that all of them have an $O(N)$ complexity while the CPU time for the computation of reaction components has a much smaller linear scaling constant {due to the fact that most of the equivalent polarization sources are well-separated with the targets.} CPU time with multiple cores is given in Table \ref{Table:ex1three} and it shows that, due to the small amount of CPU time in computing the reaction components, the speedup of the parallel computing is mainly decided by the computation of the free space components. Here, we only use parallel implementation within the computation of each component. Note the computation of each component is independent of other components. Therefore, it is straightforward to implement a version of the code which computes all components in parallel.
\begin{figure}[ht!]
	\center
	\subfigure[distribution of particles]{\includegraphics[scale=0.33]{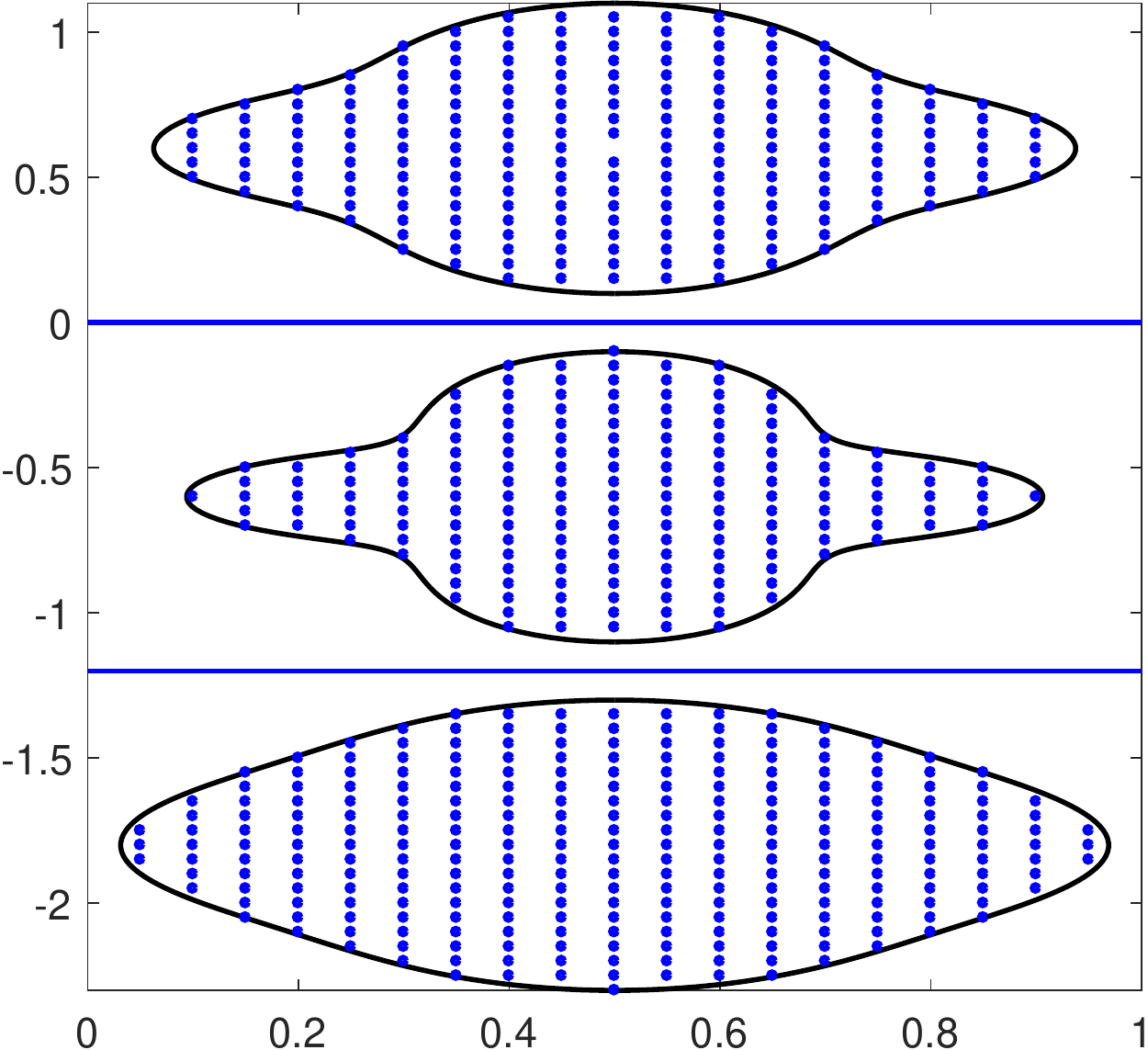}}\quad
	\subfigure[convergence rates vs. $p$]{\includegraphics[scale=0.3]{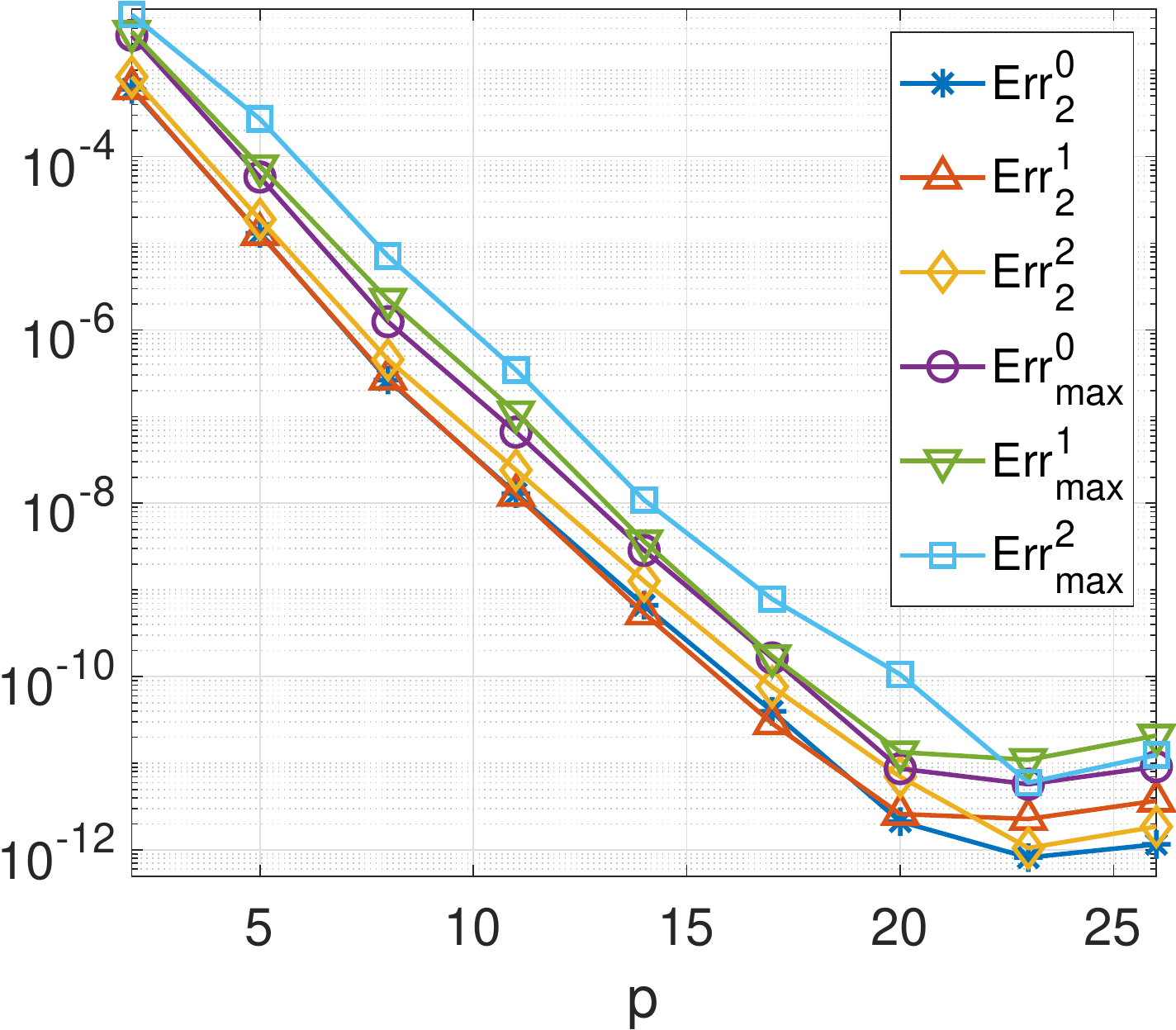}}\quad
	\subfigure[CPU time vs. $N$]{\includegraphics[scale=0.3]{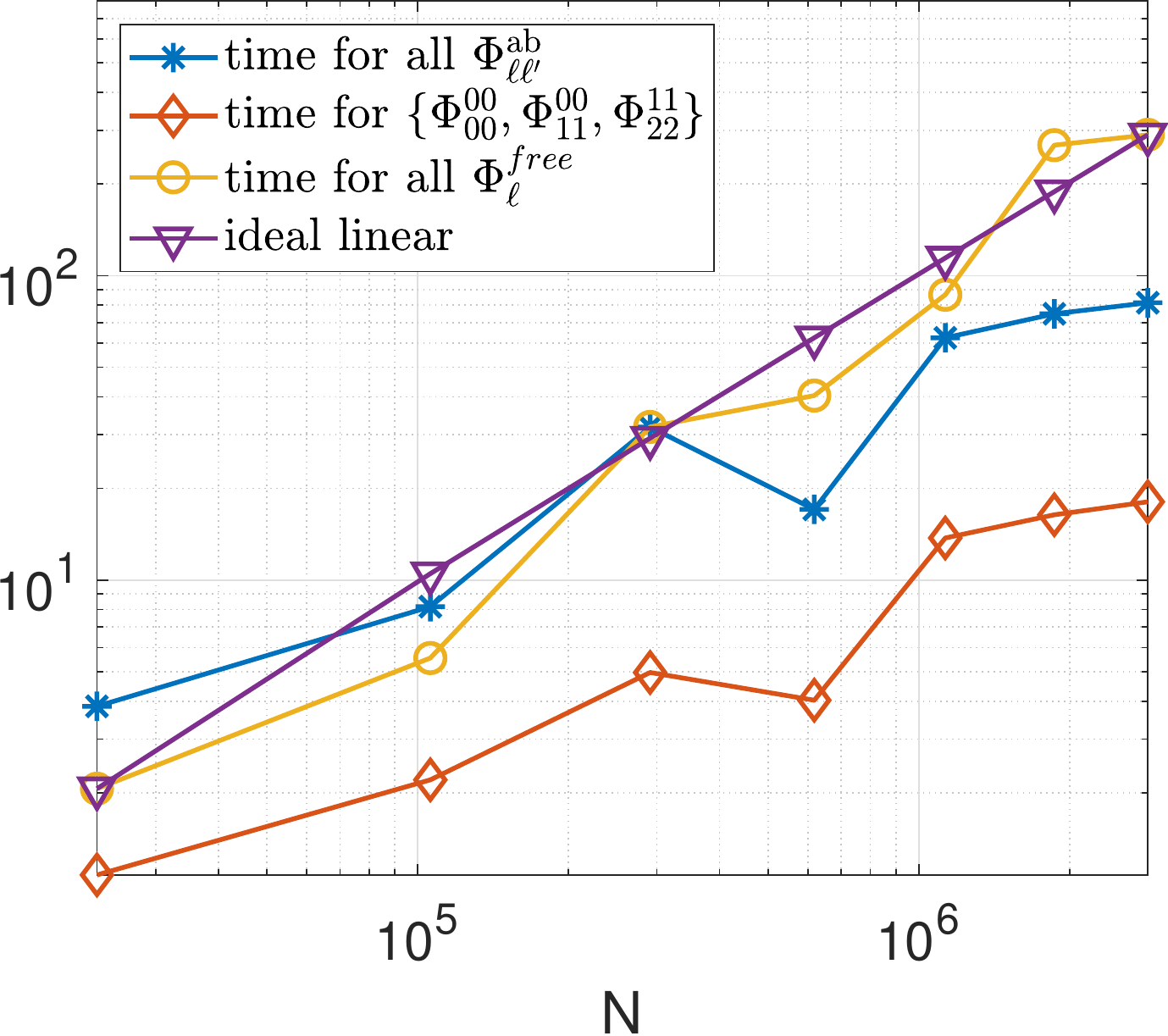}}
	\caption{Performance of FMM for a three layers media problem.}%
	\label{fmmperformance}%
\end{figure}
\begin{table}[ht!]
	\centering {\small
		\begin{tabular}
			[c]{|c|c|c|c|}\hline
			cores & $N$ & time for all $\{\Phi_{\ell}^{free}\}_{\ell=0}^2$ & time for all $\{\Phi^{\mathfrak{ab}}_{\ell\ell'}\}$ 			\\\hline
			\multirow{4}{*}{1} &618256 & 40.36 & 17.06  \\\cline{2-4}
			& 1128556 & 86.72 & 62.47  \\\cline{2-4}
			& 1862568 & 269.05 & 74.93 \\\cline{2-4}
			& 2861288 & 292.42 & 81.47 \\\hline
			\multirow{4}{*}{6} & 618256 & 7.653 & 3.613\\\cline{2-4}
			& 1128556 & 16.29 & 12.50  \\\cline{2-4}
			& 1862568 & 50.72 & 15.52  \\\cline{2-4}
			& 2861288 & 54.85 & 17.27  \\\hline
			\multirow{4}{*}{36} & 618256 & 2.042 & 1.639  \\\cline{2-4}
			& 1128556 & 4.308 &4.459  \\\cline{2-4}
			& 1862568 & 14.94 & 6.104  \\\cline{2-4}
			& 2861288 & 15.21 & 7.673  \\\hline
		\end{tabular}
	}
	\caption{Comparison of CPU time with multiple cores ($p=5$).}%
	\label{Table:ex1three}%
\end{table}

\section{Conclusion}

In this paper, we have presented a fast multipole
method associated with the Green's function of linearized Poisson-Boltzmann equation in 3-D layered media. The potential of interest has been decomposed into a free space and four types of reaction field components. By another extension of the Funk-Hecke formula involving pure imaginary wave number, we developed ME of $O(p^{2})$ terms for the far field of the reaction components,  which can be associated with polarization sources at specific locations for each type of the reaction field components. M2L translation operators are also developed for the reaction components. As a result, the traditional FMM framework can be applied to both the free space and reaction components once the polarization sources are used together with the original targets. Due to the separation of the polarization coordinates and the corresponding target positions by a material interface, the computational cost from the reaction component is only a fraction of that of the FMM for the free space component. Hence, computing the potential in layered media basically costs almost the same as that for the wave interaction in the free space.

For the future work, we will carry out error estimate of the FMM for the linearized  Poisson-Boltzmann equation in 3-D
layered media, which require an error analysis for the new MEs and M2L operators for the reaction components. The combination of the FMM with integral method for efficient simulation of ion channel will also be our next research work.

\section*{Acknowledgement}

The first author acknowledges the financial support provided by NSFC (grant 11771137) and the Construct Program of the Key Discipline in Hunan Province.
This work was supported by US Army Research Office (Grant No.
W911NF-17-1-0368).


\end{document}